\documentclass[11pt,reqno]{amsart}
\usepackage{amsmath,amssymb,mathrsfs}
\usepackage{graphicx,cite}
\usepackage{subfigure}
\usepackage{epstopdf}
\usepackage{graphics} %% add this and next lines if pictures should be in esp format
\usepackage{epsfig} %For pictures: screened artwork should be set up with an 85 or 100 line screen
\usepackage{tikz}
\usepackage{paralist}
\usepackage{booktabs}
\usepackage{enumerate}
\usepackage{enumitem}
\usepackage{mdwlist}

\newtheorem{theorem}{Theorem}[section]

\newtheorem{lemma}[theorem]{Lemma}

\newtheorem{remark}[theorem]{Remark}

\numberwithin{equation}{section}
\begin{document}

\title[A BIM for the elastic scattering problem]{A highly
accurate boundary integral method for the elastic obstacle scattering problem}

\author{Heping Dong}
\address{School of Mathematics, Jilin University, Changchun,  Jilin 130012,  China}
\email{dhp@jlu.edu.cn}

\author{Jun Lai}
\address{School of Mathematical Sciences, Zhejiang University,
	Hangzhou, Zhejiang 310027, China}
\email{laijun6@zju.edu.cn}

\author{Peijun Li}
\address{Department of Mathematics, Purdue University, West Lafayette, Indiana
	47907, USA}
\email{lipeijun@math.purdue.edu}

%\thanks{}



\begin{abstract}
Consider the scattering of a time-harmonic plane wave by a rigid obstacle
embedded in a homogeneous and isotropic elastic medium in two dimensions. In
this paper, a novel boundary integral formulation is proposed and its highly
accurate numerical method is developed for the elastic obstacle
scattering problem. More specifically, based on the Helmholtz
decomposition, the model problem is reduced to a coupled boundary integral
equation with singular kernels. A regularized system is constructed in order to
handle the degenerated integral operators. The semi-discrete and full-discrete
schemes are studied for the boundary integral system by using the trigonometric
collocation method. Convergence is established for the numerical schemes in
some appropriate Sobolev spaces. Numerical experiments are presented for both
smooth and nonsmooth obstacles to demonstrate the superior performance of the
proposed method.
\end{abstract}

\maketitle

\section{Introduction}

The phenomena of elastic scattering by obstacles have received ever-increasing
attention due to the significant applications in diverse scientific areas such
as geological exploration, nondestructive testing, and medical diagnostics
\cite{LL-86, ABG-15}. The scattering problems for elastic waves have been
extensively studied; there are many mathematical and computational results
available for both the direct and inverse scattering problems \cite{AH-SIAP76,
HKS-IP13, L-SIAP12, PV-JASA, TC2007}. It has played an important role to have an
accurate and efficient numerical method in many of these applications since a
large number of forward simulations are often required. Due to the coexistence
of compressional and shear wave components with different wavenumbers, the
propagation of elastic waves governed by the Navier equation is much more
complicated than that of acoustic waves governed by the Helmholtz equation. This
paper is concerned with the scattering of a time-harmonic plane wave by a rigid
obstacle embedded in a homogeneous and isotropic elastic medium in two
dimensions. We propose a novel boundary integral formulation and develop a
highly accurate numerical method for solving the elastic obstacle scattering
problem.

Given the importance of elasticity, various numerical methods have been
proposed to solve the associated scattering problems in the literature.
Conventional methods include the finite difference and finite element methods. 
Despite being successful to deal with media with general properties and
geometries, they require the discretization of the whole computational domain
and encounter the issue of domain truncation by adding some artificial absorbing
boundary layers \cite{GK-WM90}. The method of boundary integral equations
offers an attractive alternative for solving the exterior boundary value
problems such as the obstacle scattering problems. It only requires the
discretization of boundary of the domain and satisfies the radiation condition
exactly \cite{PV-JASA}, but it does require the knowledge of Green's function
for the governing equation. As is known, the Green function of the elastic wave
equation is a second order tensor and is complicated to be applied in the
computation of boundary integral equations. Readers are referred to
\cite{BXY-JCP17, BLR-JCP14} and reference therein for some recent advances along
this direction. To bypass this complexity, we introduce two scalar potential
functions and use the Helmholtz decomposition to split the displacement of the
elastic wave field into the compressional wave and the shear wave. The two wave
components, both of which satisfy the two-dimensional Helmholtz equation
\cite{DLL2019, LL2019, YLLY-CiCP18}, are only coupled at the boundary of
obstacle. Therefore, the boundary value problem of the Navier equation is
converted equivalently into a coupled boundary value problem of the Helmholtz
equations for the potentials. Compared to the formulation based on the second
order tensor elastic Green's function, such a decomposition reduce greatly the
complexity for the computation of the elastic scattering problem. Similar
techniques have also been successfully applied to the equations of unsteady and
incompressible flow \cite{GJ2019}.

Another goal of this work is to carry on the convergence analysis of a high
order numerical discretization for the boundary integral system. Numerical
discretization for boundary integral equations requires special quadratures due
to singular integral kernels \cite{Alpert1999}. The quadrature methods for the
logarithmic and hypersingular integral equations were proposed in \cite{KS1993,
Kress1995} to solve the sound-soft and sound-hard obstacle scattering problems,
where the error analysis was done in the Sobolev space and H\"{o}lder space,
respectively. As an improvement of the quadrature method for the sound-hard
obstacle scattering problem, based on the trigonometric differentiation to
discrete the principal part of the hypersingular operator, a fully discrete
collocation method was proposed in \cite{Kress2014} and the convergence in a
Sobolev space setting was also proven. In \cite{SV1996}, the authors showed an
error analysis by using the trigonometric collocation method for the boundary
integral equation which contains more general singular integral operators. A
Galerkin boundary element method with a regularization for the hypersingular
integral was developed in \cite{BXY-JCP17} to solve the two-dimensional elastic
scattering problem. In \cite{KirschRitter1999}, the problem for bending of an
elastic plate with the Dirichlet boundary conditions was studied. An explicit
equivalent regularizer was constructed for the Fredholm integral equation of
second kind to derive existence and uniqueness results in an appropriate Sobolev
space. A high order spectral algorithm was developed in \cite{L2014} for the
three-dimensional elastic obstacle scattering problem with the Dirichlet or
Neumann boundary condition. A Nystr\"{o}m method with a local correction scheme
was shown in \cite{TC2007} for the elastic obstacle scattering problem in three
dimensions. We refer to \cite{MP-book1986} for a comprehensive account of 
the singular integral equations.

In this work, by using the Helmholtz decomposition, the exterior boundary
value problem of the elastic obstacle scattering is reduced to a coupled
boundary integral equation with the Cauchy type singular integral
operators. Based on the recent works \cite{DLL2019, DLL2020, LL2019},
we introduce an appropriate regularizer to the boundary integral system and
split the singular integral operator in the form of an isomorphic operator plus
a compact one, which enables us to derive the convergence result in some Sobolev
spaces. The semi-discrete and full-discrete schemes are examined for the
boundary integral system under the framework of trigonometric collocation
method. We deduce the error estimates for both the semi- and full-discrete
schemes and show that the numerical solution of the integral system
converges to the exact solution. In particular, we demonstrate that the proposed
scheme converges exponentially fast when the boundary of the obstacle and the
incident wave are analytic. Numerical experiments for both smooth and nonsmooth
obstacles are provided to confirm our theoretical analysis. We point out that
the proposed method is able to achieve a very high precision even for boundaries
with corners by using the graded meshes \cite{AOT2012, DR-shu2, Kress1990}. It
is also worth mentioning that our method is extremely fast since the
full-discrete scheme is established via simple quadrature operators. Most
importantly, we only need to solve the scalar Helmholtz equation instead of
solving the vector Navier equation. This feature makes the approach particularly
attractive as anyone, who has the code available to solve the acoustic obstacle
scattering problem, is able to adapt to solve the elastic obstacle scattering
problem. The application of this formulation to elastic
multi-particle scattering with fast multipole method and inverse elastic obstacle problem have been
thoroughly investigated in \cite{LL2019, DLL2019}.

This paper concerns both the theoretical analysis and numerical computation for
the elastic obstacle scattering problem. The work contains three contributions: 

\begin{enumerate}
\item propose a novel boundary integral formulation by introducing a regularizer
to the integral system obtained by applying the Helmholtz decomposition to the
Navier equation; 

\item  establish the convergence of the semi- and full-discrete schemes of the
boundary integral system via the trigonometric collocation method;

\item demonstrate the superior numerical performance by presenting examples of
smooth and nonsmooth obstacles.
\end{enumerate}

The paper is organized as follows. In Section 2, we introduce the problem
formulation. Section 3 presents the boundary integral equations, gives the
decomposition of integral operators, and deduces an operator equation in
form of an isomorphic operator plus a compact one. Section 4 is devoted to the
convergence analysis of the semi-discrete and full-discrete schemes for the
boundary integral system via the trigonometric collocation method. Numerical
experiments are presented to verify the theoretical findings in Section 5. The
paper is concluded with some general remarks and discussions on the future work
in Section 6.

\section{Problem formulation}

Consider a two-dimensional elastically rigid obstacle, which is described as a
bounded domain $D\subset\mathbb R^2$ with an analytic boundary $\Gamma_D$.
Denote by $\nu=(\nu_1, \nu_2)^\top$ and $\tau=(\tau_1, \tau_2)^\top$ the unit
normal and tangential vectors on $\Gamma_D$, respectively, where $\tau_1=-\nu_2,
\tau_2=\nu_1$. The exterior domain $\mathbb{R}^2\setminus \overline{D}$ is
assumed to be filled with a homogeneous and isotropic elastic medium with a unit
mass density.

Let the obstacle be illuminated by a time-harmonic compressional plane wave
$\boldsymbol{u}^{\rm inc}(x)=d \mathrm{e}^{\mathrm{i} \kappa_{\mathfrak p}d\cdot
x}$ or shear plane wave $\boldsymbol{u}^{\rm inc}(x)=d^{\perp} 
\mathrm{e}^{\mathrm{i}\kappa_{\mathfrak s} d\cdot x}$,
where $d=(\cos\theta, \sin\theta)^\top$ is
the unit propagation direction vector, $\theta\in [0, 2\pi)$ is the incident
angle, $d^{\perp}=(-\sin\theta, \cos\theta)^\top$ is an orthonormal vector of
$d$,
and
\[
\kappa_{\mathfrak p}=\frac{\omega}{\sqrt{\lambda+2\mu}},\quad
\kappa_{\mathfrak s}=\frac{\omega}{\sqrt{\mu}}
\]
are the compressional wavenumber and the shear wavenumber, respectively.

The displacement of the total field $\boldsymbol{u}$ satisfies the Navier
equation
\[
\mu\Delta\boldsymbol{u}+(\lambda+\mu)\nabla\nabla\cdot\boldsymbol{u}
+\omega^2\boldsymbol{u}=0\quad {\rm in}\, \mathbb{R}^2\setminus \overline{D}, 	
\]
where $\omega>0$ is the angular frequency and $\lambda, \mu$ are the
Lam\'{e} constants satisfying $\mu>0, \lambda+\mu>0$. Since the obstacle is
assumed to be rigid, the total field $\boldsymbol u$ satisfies the homogeneous
boundary condition 
\[
\boldsymbol{u}=0\quad {\rm on}~\Gamma_D.
\]
The total field $\boldsymbol u$ consists of the incident field $\boldsymbol
u^{\rm inc}$ and the scattered field $\boldsymbol v$, i.e., 
\[
\boldsymbol u=\boldsymbol u^{\rm inc}+\boldsymbol v. 
\]
It is easy to verify that the scattered field $\boldsymbol v$ satisfies
the boundary value problem
\begin{equation}\label{scatteredfield}
\begin{cases}
\mu\Delta\boldsymbol{v}+(\lambda+\mu)\nabla\nabla\cdot\boldsymbol{v}
+\omega^2\boldsymbol{v}=0\quad &{\rm in}~
\mathbb{R}^2\setminus\overline{D},\\
\boldsymbol{v}=-\boldsymbol{u}^{\rm inc}\quad &{\rm on}~\Gamma_D.
\end{cases}
\end{equation}
In addition, the scattered field $\boldsymbol v$ is required to satisfy
the Kupradze--Sommerfeld radiation condition
\[
\lim_{\rho\to\infty}\rho^{\frac{1}{2}}(\partial_\rho\boldsymbol v_{\mathfrak p}-\mathrm{i}\kappa_{\mathfrak p}\boldsymbol v_{\mathfrak p})=0,\quad
\lim_{\rho\to\infty}\rho^{\frac{1}{2}}(\partial_\rho\boldsymbol v_{\mathfrak s}-\mathrm{i}\kappa_{\mathfrak s}\boldsymbol v_{\mathfrak s})=0,\quad \rho=|x|,
\]
where
\[
\boldsymbol v_{\mathfrak p}=-\frac{1}{\kappa_{\mathfrak p}^2}\nabla\nabla\cdot\boldsymbol
v,\quad \boldsymbol v_{\mathfrak s}=\frac{1}{\kappa_{\mathfrak s}^2}{\bf curl}{\rm
curl}\boldsymbol v,
\]
are known as the compressional and shear wave components of $\boldsymbol v$,
respectively. Given a vector function $\boldsymbol v=(v_1, v_2)^\top$ and
a scalar function $v$, the scalar and vector curl operators are defined by
\[
{\rm curl}\boldsymbol v=\partial_{x_1}v_2-\partial_{x_2}v_1, \quad
{\bf curl}v=(\partial_{x_2}v, -\partial_{x_1}v)^\top.
\]

For any solution $\boldsymbol v$ of the elastic wave equation
\eqref{scatteredfield}, the Helmholtz decomposition reads
\begin{equation}\label{HelmDeco}
\boldsymbol{v}=\nabla\phi+\boldsymbol{\rm curl}\psi,
\end{equation}
where $\phi,\psi$ are two scalar functions. Combining
\eqref{scatteredfield} and \eqref{HelmDeco} yields the Helmholtz equations
\[
\Delta\phi+\kappa_{\mathfrak p}^{2}\phi=0, \quad \Delta\psi+\kappa_{\mathfrak s}^{2}\psi=0.
\]
As usual, $\phi$ and $\psi$ are required to satisfy the Sommerfeld radiation
conditions
\begin{equation*}
\lim_{\rho\to\infty}\rho^{\frac{1}{2}}(\partial_{\rho}
\phi-\mathrm{i}\kappa_{\mathfrak p}\phi)=0, \quad
\lim_{\rho\to\infty}\rho^{\frac{1}{2}}(\partial_{\rho}
\psi-\mathrm{i}\kappa_{\mathfrak s}\psi)=0, \quad \rho=|x|.
\end{equation*}
It follows from the Helmholtz decomposition and the boundary condition on
$\Gamma_D$ that
\[
\boldsymbol{v}=\nabla\phi+{\bf curl}\psi=-\boldsymbol{u}^{\rm
inc}.
\]
Taking the dot product of the above equation with $\tau$ and $\nu$,
respectively, we get
\[
\partial_\nu\phi+\partial_\tau\psi=f_1,\quad
\partial_\tau\phi-\partial_\nu\psi=f_2,
\]
where
\[
f_1=-\nu\cdot\boldsymbol{u}^{\rm inc},\quad
f_2=-\tau\cdot\boldsymbol{u}^{\rm inc}.
\]
In summary, the scalar potential functions $\phi, \psi$ satisfy the coupled
boundary value problem
\begin{align}\label{HelmholtzDec}
\begin{cases}
\Delta\phi+\kappa_{\mathfrak p}^{2}\phi=0, \quad \Delta\psi+\kappa_{\mathfrak s}^{2}\psi=0,
\quad &{\rm in} ~\mathbb{R}^2\setminus\overline{D},\\
\partial_\nu\phi+\partial_\tau\psi=f_1, \quad
\partial_\tau\phi-\partial_\nu\psi=f_2, \quad &{\rm on} ~ \Gamma_D,\\
\displaystyle{\lim_{\rho\to\infty}\rho^{\frac{1}{2}}(\partial_{\rho}\phi-\mathrm{i}\kappa_{\mathfrak p}\phi)=0}, \quad
\displaystyle{\lim_{\rho\to\infty}\rho^{\frac{1}{2}}(\partial_{\rho}\psi-\mathrm{i}\kappa_{\mathfrak s}\psi)=0}, \quad &\rho=|x|.
\end{cases}
\end{align}

It is well known that a radiating solution of \eqref{scatteredfield} has the
asymptotic behavior of the form
\begin{equation}\label{farf}
\boldsymbol{v}(x)=\frac{\mathrm{e}^{\mathrm{i}\kappa_{\mathfrak p}|x|}}{\sqrt{|x|}}\boldsymbol{v}_{\mathfrak p}^\infty(\hat{x})+\frac{\mathrm{e}^{\mathrm{i}\kappa_{\mathfrak s}|x|}}{\sqrt{|x|}}\boldsymbol{v}_{\mathfrak s}^\infty(\hat{x})+\mathcal{O}\left(\frac{1}{|x|^{\frac{3}{2}}}\right),
\quad |x|\to\infty
\end{equation}
uniformly in all directions $\hat{x}:=x/|x|$, where $\boldsymbol{v}_{\mathfrak
p}^\infty$ and $\boldsymbol{v}_{\mathfrak s}^\infty$, defined on
the unit circle $\Omega$, are known as the compressional and shear 
far-field patterns of $\boldsymbol{v}$, respectively. The following result
presents the relationship between the compressional (or shear) far-field pattern
of $\boldsymbol{v}$ and far-field patterns of $\phi$ (or $\psi$). The proof may
be found in \cite{DLL2019}. 

\begin{lemma}\label{Th1}
The far-field pattern \eqref{farf} for the radiating solution $\boldsymbol{v}$ to the Navier equation satisfies
\begin{equation}\label{behaviour relation}
\boldsymbol{v}_{\mathfrak p}^\infty(\hat{x}) =\mathrm{i}\kappa_{\mathfrak p}\phi_{\infty}(\hat{x})\hat{x}, \qquad \boldsymbol{v}_{\mathfrak s}^\infty(\hat{x})=-\mathrm{i}\kappa_{\mathfrak	s}\psi_{\infty}(\hat{x})\hat{x}^{\perp},
\end{equation}
where the complex-valued functions $\phi_\infty(\hat{x})$ and
$\psi_\infty(\hat{x})$ are the far-field patterns corresponding to $\phi$ and
$\psi$, respectively. 
\end{lemma}

By Lemma \ref{Th1} and the Helmholtz decomposition, it is clear that the elastic
scattered field $\boldsymbol{v}_{\mathfrak p}, \boldsymbol{v}_{\mathfrak s}$
and the corresponding far-field patterns $\boldsymbol{v}_{\mathfrak
p}^\infty, \boldsymbol{v}_{\mathfrak s} ^\infty$ can be obtained by solving the
coupled boundary value problem \eqref{HelmholtzDec}.

%%%%%%%%%%%%%%%%%%%%%%%%%%%%%%%%%%%%%%%%%%%%%%%%
%%%%%%%%%%%%%%%%%%%%%%%%%%%%%%%%%%%%%%%%%%%%%%%%

\section{Boundary integral equations}

In this section, a novel boundary integral formulation is proposed for the
coupled boundary value problem \eqref{HelmholtzDec}. In particular, a 
regularizer is constructed in order to handle the degenerated integral
operators.

\subsection{Coupled integral equations}

Denote the fundamental solution to the Helmholtz equation in two dimensions by
$$
\Phi(x,y;\kappa)=\frac{\mathrm{i}}{4}H_0^{(1)}(\kappa|x-y|), \quad x\neq y,
$$
where $H_0^{(1)}$ is the Hankel function of the first kind with order zero. 
We assume that the solution of \eqref{HelmholtzDec} is given as the following
single-layer potentials with densities $g_1, g_2$:
\begin{align}
\phi(x)=\int_{\Gamma_D}\Phi(x,y;\kappa_{\mathfrak p})g_1(y)\,\mathrm{d}s(y), \quad 
\psi(x)=\int_{\Gamma_D}\Phi(x,y;\kappa_{\mathfrak s})g_2(y)\,\mathrm{d}s(y), \quad 
x\in\mathbb{R}^2\setminus\Gamma_D. \label{singlelayer}
\end{align}

Letting $x\in\mathbb{R}^2\setminus\overline{D}$ approach the boundary $\Gamma_D$
in \eqref{singlelayer}, and using the jump relation of single-layer potentials
and the boundary condition of \eqref{HelmholtzDec}, we deduce for
$x\in\Gamma_D$ that 
\begin{align}
\begin{split}\label{boundaryIE}
-\frac{1}{2}g_1(x)+\int_{\Gamma_D}\frac{\partial\Phi(x,y;\kappa_{\mathfrak p})}
{\partial\nu(x)}g_1(y)\,\mathrm{d}s(y)+\int_{\Gamma_D}
\frac{\partial\Phi(x,y;\kappa_{\mathfrak s})}
{\partial\tau(x)}g_2(y)\,\mathrm{d}s(y)=-\nu(x)\cdot\boldsymbol{u}^{\rm inc}(x),\\
\int_{\Gamma_D}\frac{\partial\Phi(x,y;\kappa_{\mathfrak p})}
{\partial\tau(x)}g_1(y)\,\mathrm{d}s(y)+ \frac{1}{2}g_2(x)-\int_{\Gamma_D}
\frac{\partial\Phi(x,y;\kappa_{\mathfrak s})}
{\partial\nu(x)}g_2(y)\,\mathrm{d}s(y)=-\tau(x)\cdot\boldsymbol{u}^{\rm inc}(x). 
\end{split}
\end{align}
The corresponding far-field patterns can be represented by 
\begin{align} \label{singlelayer_far}
\phi_\infty(\hat{x})=\gamma_{\mathfrak p}\int_{\Gamma_D}\mathrm{e}^{-\mathrm{i}
\kappa_{\mathfrak p}\hat{x}\cdot y} g_1(y)\,\mathrm{d}s(y), \quad
\psi_\infty(\hat{x})=\gamma_{\mathfrak s}\int_{\Gamma_D}
\mathrm{e}^{-\mathrm{i}\kappa_{\mathfrak s} \hat{x}\cdot y} g_2(y)\,\mathrm{d}s(y),
\quad\hat{x}\in\Omega,
\end{align}
where $\gamma_\sigma= e^{\mathrm{i}\pi/4}/{\sqrt{8\kappa_\sigma\pi}}$ for
$\sigma={\mathfrak p}$ or ${\mathfrak s}$.

We introduce the single-layer integral operator and the corresponding far-field
integral operator expressed by
\[
(S^\sigma g)(x)=2\int_{\Gamma_D}\Phi(x,y;\kappa_\sigma)g(y)\,\mathrm{d}s(y), 
\quad (S^\sigma_\infty
g)(\hat{x})=\gamma_{\sigma}\int_{\Gamma_D}\mathrm{e}^{-\mathrm{i }\kappa_\sigma
\hat{x}\cdot y}g(y)\,\mathrm{d}s(y), \quad x\in\Gamma_D,\,\hat{x}\in\Omega.
\]
In addition, we introduce the normal derivative and the tangential
derivative boundary integral operators
\[
(K^\sigma g)(x)=2\int_{\Gamma_D}\frac{\partial\Phi(x,y;\kappa_\sigma)}
{\partial\nu(x)}g(y)\,\mathrm{d}s(y), 
\quad
(H^\sigma g)(x)=2\int_{\Gamma_D}\frac{\partial\Phi(x,y;\kappa_\sigma)}
{\partial\tau(x)}g(y)\,\mathrm{d}s(y), \quad x\in\Gamma_D.
\]
Note that the operators
$K^\sigma$ and $H^\sigma$ are defined in the sense of Cauchy principal value.
Based on the boundary integral operators, the coupled boundary integral
equations \eqref{boundaryIE} can be written into the operator
form
\begin{eqnarray}
\left\{
\begin{split}\label{direct field}
-g_1+K^{\mathfrak p}g_1+H^{\mathfrak s}g_2&=2f_1, \\ 
g_2+H^{\mathfrak p}g_1-K^{\mathfrak s}g_2&=2f_2.
\end{split}
\right.
\end{eqnarray}
Once the system \eqref{direct field} is solved for the densities $g_1$ and
$g_2$, the corresponding far-field patterns of \eqref{singlelayer_far} can be
represented as follows 
\begin{align}\label{farfield}
\phi_\infty(\hat{x})=(S^{\mathfrak p}_\infty g_1)(\hat{x}), \quad
\psi_\infty(\hat{x})=(S^{\mathfrak s}_\infty g_2)(\hat{x}),
\qquad\hat{x}\in\Omega. 
\end{align}

By \cite[Theorems 4.1 and 4.8]{LL2019}, the coupled system \eqref{direct field}
is uniquely solvable if neither $\kappa_{\mathfrak p}$ nor $\kappa_{\mathfrak
s}$ is the eigenvalue of the interior Dirichlet problem for the Helmholtz
equation in $D$. Throughout, we assume that this condition is satisfied so that
the system \eqref{direct field} admits a unique solution.

%------------------------------------------------

\subsection{Decomposition of the operators}

We assume that the boundary $\Gamma_D$ is an analytic curve with the parametric form
\begin{equation*}%\label{obstacle}
\Gamma_D=\{z(t)=(z_1(t),z_2(t)): 0\leq t<2\pi\}, 
\end{equation*}
where $z: \mathbb{R}\rightarrow\mathbb{R}^2$
is analytic and $2\pi$-periodic with $|z'(t)|>0$ for all $t$. The parameterized
integral operators are still denoted by
$S^\sigma$, $S^\sigma_\infty$, $K^\sigma$, and $H^\sigma$ for convenience, i.e.,
\begin{align*}
(S^\sigma\varphi)(t)&=\frac{\mathrm{i}}{2}\int_0^{2\pi}
H_0^{(1)}(\kappa_\sigma|z(t)-z(\varsigma)|)\varphi(\varsigma)\,\mathrm{d}\varsigma,
&&(K^\sigma\varphi)(t)=\int_0^{2\pi}k^\sigma(t,\varsigma)\varphi(\varsigma)\,\mathrm{d}\varsigma,\\
(S_\infty^\sigma\varphi)(t)&=\gamma_{\sigma}\int_0^{2\pi}\mathrm{e}^{-\mathrm{i}\kappa_\sigma\hat{x}(t)\cdot	z(\varsigma)}\varphi(\varsigma)\,\mathrm{d}\varsigma,
&&(H^\sigma\varphi)(t)=\int_0^{2\pi}h^\sigma(t,\varsigma)\varphi(\varsigma)\,\mathrm{d}\varsigma,
\end{align*}
where 
\begin{align*}
k^\sigma(t,\varsigma)&=\frac{\mathrm{i}\kappa_\sigma}{2}\mathsf{n}(t)\cdot[z(\varsigma)-z(t)] \frac{H_1^{(1)}(\kappa_\sigma|z(t)-z(\varsigma)|)}{|z(t)-z(\varsigma)|},\\
h^\sigma(t,\varsigma)&=\frac{\mathrm{i}\kappa_\sigma}{2}\mathsf{n}^\perp(t)\cdot[z(\varsigma)-z(t)] \frac{H_1^{(1)}(\kappa_\sigma|z(t)-z(\varsigma)|)}{|z(t)-z(\varsigma)|},
\end{align*}
and 
\begin{align*}
\mathsf{n}(t)&\overset{\text{def}}{=}\tilde{\nu}(t)|z'(t)|=\big(z'_2(t), -z'_1(t)\big)^\top,
\quad\tilde{\nu}=\nu\circ z,\\
\mathsf{n}^\perp(t)&\overset{\text{def}}{=}\tilde{s}(t)|z'(t)|=\big(z'_1(t), z'_2(t)\big)^\top,
\quad\tilde{s}=\tau\circ z.
\end{align*}
Multiplying $|z'|$ on both sides of \eqref{direct field}, we obtain the
parametric form
\begin{align}\label{direct parafield}
\mathcal{A}\varphi\overset{\text{def}}{=}\left[     \begin{array}{cc}
-I+K^{\mathfrak p} & H^{\mathfrak s} \\ 
H^{\mathfrak p}    & I-K^{\mathfrak s}
\end{array}
\right]\left[                  
\begin{array}{c}
\varphi_1 \\ 
\varphi_2
\end{array}
\right]=\left[                  
\begin{array}{c}
w_1 \\ 
w_2
\end{array}
\right],
\end{align}
where $w_j=2(f_j\circ z)|z'|$, $\varphi_j=(g_j\circ z)|z'|$, $j=1,
2$, and $I$ is the identity operator.

The kernel $k^\sigma(t,\varsigma)$ of the parameterized normal derivative integral
operator can be written as
\[
k^\sigma(t,\varsigma)=k_1^\sigma(t,
\varsigma)\ln\Big(4\sin^2\frac{t-\varsigma}{2}\Big)+k_2^\sigma(t,\varsigma),
\]
where
\begin{align*}
k_1^\sigma(t,\varsigma)&= \frac{\kappa_\sigma}{2\pi}\mathsf{n}(t)\cdot\big[
z(t)-z(\varsigma)\big]\frac{J_1(\kappa_\sigma|z(t)-z(\varsigma)|)}{
|z(t)-z(\varsigma)|}, \\
k_2^\sigma(t,\varsigma)&=k^\sigma(t,\varsigma)-k_1^\sigma(t,
\varsigma)\ln\Big(4\sin^2\frac{t-\varsigma}{2}\Big)
\end{align*}
are analytic with diagonal entries given by 
\[
k_1^\sigma(t,t)=0, \qquad k_2^\sigma(t,t)=
\frac{1}{2\pi}\frac{\mathsf{n}(t)\cdot z''(t)}{|z'(t)|^2}.
\]
Hence, $K^\sigma\varphi$ can be equivalently rewritten as 
\begin{align*}
(K^\sigma\varphi)(t)=(K^\sigma_1\varphi)(t)+(K^\sigma_2\varphi)(t)\overset{\text{def}}{=}
\int_0^{2\pi}\ln\Big(4\sin^2\frac{t-\varsigma}{2}\Big)k_1^\sigma(t,\varsigma)\varphi(\varsigma)\,\mathrm{d}\varsigma+\int_0^{2\pi}k_2^\sigma(t,\varsigma)\varphi(\varsigma)\,\mathrm{d}\varsigma.
\end{align*}

Following \cite{DLL2019}, we split the kernel $h^\sigma(t,\varsigma)$
of the parameterized tangential derivative integral operator into 
\begin{align}\label{split}
h^\sigma(t,\varsigma)= h_1(t,\varsigma)\cot\frac{\varsigma-t}{2}+h_2^\sigma(t,\varsigma)\ln\Big(4\sin^2\frac{t-\varsigma}{2}\Big)+h_3^\sigma(t,\varsigma),
\end{align}
where 
\begin{align*}
h_1(t,\varsigma)&= \frac{1}{\pi}n^\perp(t)\cdot\big[z(\varsigma)-z(t)\big]\frac{
\tan\frac{\varsigma-t}{2}}{|z(t)-z(\varsigma)|^2}, \\
h_2^\sigma(t,\varsigma)&= \frac{\kappa_\sigma}{2\pi}n^\perp(t)\cdot\big[
z(t)-z(\varsigma)\big]\frac{J_1(\kappa_\sigma|z(t)-z(\varsigma)|)}{
|z(t)-z(\varsigma)|},\\
h_3^\sigma(t,\varsigma)&=h^\sigma(t,\varsigma)-h_1^\sigma(t,\varsigma)\cot\frac{
\varsigma-t}{2}-h_2^\sigma(t,\varsigma)\ln\Big(4\sin^2\frac{t-\varsigma}{2}\Big)
\end{align*}
are analytic with diagonal entries given by 
\[
h_1(t,t)=\frac{1}{2\pi}, \quad h_2^\sigma(t,t)=0,
\quad h_3^\sigma(t,t)=0.
\]
In order to show the convergence, based on \eqref{split}, we split the
singular integral operator $H^\sigma$ into
\begin{align}\label{decomposition}
H^\sigma=H_1+E^\sigma H_2+\widetilde{H}_1+\widetilde{H}_2^\sigma+\widetilde{H}_3^\sigma,
\end{align}
where $E^\sigma\psi=\kappa_\sigma^2|z'|^2\psi$, and
\begin{align*}
(H_1\psi)(t)&=\frac{1}{2\pi}\int_0^{2\pi}\cot\frac{\varsigma-t}{2}\psi(\varsigma)\,\mathrm{d}\varsigma+\frac{\mathrm{i}}{2\pi}\int_{0}^{2\pi}\psi(\varsigma)\,\mathrm{d}\varsigma,\\
%%%%%%%%%%%%%%%%%%%%%%%%%%%%%%%%%%%%%%%%
(H_2\psi)(t)&=\frac{1}{4\pi}\int_0^{2\pi}\ln\Big(4\sin^2\frac{t-\varsigma}{2}\Big)\sin(t-\varsigma)\psi(\varsigma)\,\mathrm{d}\varsigma+\frac{\mathrm{i}}{2\pi}\int_{0}^{2\pi}\psi(\varsigma)\,\mathrm{d}\varsigma,\\
%%%%%%%%%%%%%%%%%%%%%%%%%%%%%%%%%%%%%%%%%%%%%%%%
(\widetilde{H}_1\psi)(t)&=\int_0^{2\pi}\widetilde{h}_1(t,\varsigma)\psi(\varsigma)\,\mathrm{d}\varsigma,\quad
%%%%%%%%%%%%%%%%%%%%%%%%%%%%%%%%%%%%%
(\widetilde{H}^\sigma_2\psi)(t)=\int_0^{2\pi}\ln\Big(4\sin^2\frac{t-\varsigma}{2}\Big)\widetilde{h}_2^\sigma(t,\varsigma)\psi(\varsigma)\,\mathrm{d}\varsigma,\\
%%%%%%%%%%%%%%%%%%%%%%%%%%%%%%%%%%%%%%%%%%%%%%
(\widetilde{H}^\sigma_3\psi)(t)&=\int_0^{2\pi}\widetilde{h}_3^\sigma(t,\varsigma)\psi(\varsigma)\,\mathrm{d}\varsigma.
\end{align*}
Here, the functions
\begin{align*}
&\widetilde{h}_1(t,\varsigma)=\cot\frac{\varsigma-t}{2}\big(h_1(t,\varsigma)-\frac{1}{2\pi}\big),\\
&\widetilde{h}_2^\sigma(t,\varsigma)=h_2^\sigma(t,\varsigma)-{\kappa^2_\sigma}/(4\pi)|z'(t)|^2\sin(t-\varsigma),\\
&\widetilde{h}_3^\sigma(t,\varsigma)=h_3^\sigma(t,\varsigma)-\mathrm{i}\frac{
\kappa_\sigma^2|z'(t)|^2+1}{2\pi}
\end{align*}
are analytic with diagonal entries 
$\widetilde{h}_1(t,t)=\widetilde{h}_2^\sigma(t,t)=0$. We refer to the proof of Theorem \ref{compact} for the analyticity of $\widetilde{h}_1$.

\subsection{Operator equations}

We reformulate the parametrized integral equations \eqref{direct parafield} into
a single operator form
\begin{align}\label{Matrixequation}
\mathcal{A}\varphi=(\mathcal{H}+\mathcal{B})\varphi = w,
\end{align}
where $\varphi=(\varphi_1,\varphi_2)^\top$, $w=(w_1, w_2)^\top$ and
\begin{align*}
\mathcal{H}&=\left[                  
\begin{array}{cc}
-I   & H_1 \\ 
H_1  & I 
\end{array}
\right]+\left[                  
\begin{array}{cc}
0                    & E^{\mathfrak s}H_2 \\ 
E^{\mathfrak p}H_2   & 0 
\end{array}
\right]\overset{\text{def}}{=}\mathcal{H}_1+\mathcal{H}_2,\\
%%%%%%%%%%%%%%%%%%%%%%%%%%%%%%%%%%%%%%%%%%%%%%%%%
\mathcal{B}&=\mathcal{B}_1+\mathcal{B}_2+\mathcal{B}_3\overset{\text{def}}{=}\left[                  
\begin{array}{cc}
K_1^{\mathfrak p} & \widetilde{H}_2^{\mathfrak s} \\ 
\widetilde{H}_2^{\mathfrak p} & -K_1^{\mathfrak s}
\end{array}
\right] + \left[                  
\begin{array}{cc}
K_2^{\mathfrak p}   & \widetilde{H}_3^{\mathfrak s} \\ 
\widetilde{H}_3^{\mathfrak p}   & -K_2^{\mathfrak s}
\end{array}
\right]+\left[                  
\begin{array}{cc}
0    & \widetilde{H}_1 \\ 
\widetilde{H}_1  & 0 
\end{array}
\right].
\end{align*}
More specifically, we have 
\begin{align*}
(\mathcal{B}_1\varphi)(t)&=\int_{0}^{2\pi}\ln\Big(4\sin^2\frac{t-\varsigma}{2}
\Big)\left[\begin{array}{cc}
k_1^{\mathfrak p}(t,\varsigma)& \widetilde{h}_2^{\mathfrak s}(t,\varsigma) \\ 
\widetilde{h}_2^{\mathfrak p}(t,\varsigma)& 
-k_1^{\mathfrak s}(t,\varsigma)
\end{array}
\right]\left[                  
\begin{array}{c}
\varphi_1(\varsigma) \\ 
\varphi_2(\varsigma)
\end{array}
\right]\,\mathrm{d}\varsigma, \\
%%%%%%%%%%%%%%%%%%%%%%%%%%%%%%%%%%%%%%%%%%%%%%%%%
(\mathcal{B}_2\varphi+\mathcal{B}_3\varphi)(t)&=
\int_{0}^{2\pi}
\left[                  
\begin{array}{cc}
k_2^{\mathfrak p}(t,\varsigma)& \widetilde{h}_3^{\mathfrak s}(t,\varsigma)+\widetilde{h}_1(t,\varsigma) \\ 
\widetilde{h}_3^{\mathfrak p}(t,\varsigma)+\widetilde{h}_1(t,\varsigma)& 
-k_2^{\mathfrak s}(t,\varsigma)
\end{array}
\right]\left[                  
\begin{array}{c}
\varphi_1(\varsigma) \\ 
\varphi_2(\varsigma)
\end{array}
\right]\,\mathrm{d}\varsigma.
\end{align*}

Let $H^p[0,2\pi], p\geq0$ denote the space of $2\pi$-periodic functions $u: \mathbb{R}\rightarrow \mathbb{C}$ equipped with the norm
$$
\|u\|_p^2:=\sum_{m=-\infty}^\infty(1+m^2)^p|\hat{u}_m|^2<\infty,
$$
where
$$
\hat{u}_m=\frac{1}{2\pi}\int_0^{2\pi}u(t)e^{-\mathrm{i}mt}
\,\mathrm{d}t, \qquad m=0,\pm1,\pm2,\cdots
$$
are the Fourier coefficients of $u$. Define Sobolev spaces
\begin{align*}
H^p[0,2\pi]^2&=\bigg\{w=(w_1, w_2)^\top;w_1(t)\in H^p[0,2\pi],w_2(t)\in H^p[0,2\pi]\bigg\},
\\
H^p_*[0,2\pi]^2&=\bigg\{w=(w_1, w_2)^\top;w(t)\in H^p[0,2\pi]^2,
(\mathcal{H}_1w)(t)\in H^{p+2}[0,2\pi]^2\bigg\},
\end{align*}
which are equipped with the norms 
\begin{align}\label{norms}
\begin{split}
\|w\|_p&=\|w_1\|_p+\|w_2\|_p,\\
\|w\|_{p,*}&=\|w_1\|_{p}+\|w_2\|_{p}+\|H_1w_2-w_1\|_{p+2}+\|H_1w_1+w_2\|_{p+2}.
\end{split}
\end{align}
It is easy to see the embedding relation $H^{p+2}[0,2\pi]^2\hookrightarrow
H^p_*[0,2\pi]^2\hookrightarrow H^p[0,2\pi]^2$ since $H_1:
H^p[0,2\pi]\rightarrow H^{p}[0,2\pi]$ is bounded (cf. Theorem \ref{Hbounded}). 

It is difficult to analyze directly the operator equation \eqref{Matrixequation}
since the leading term $\mathcal{H}_1$ is degenerated \cite{LL2019}. To overcome
this difficulty, we introduce a regularizer via multiplying both sides of
\eqref{Matrixequation} by the operator $\mathcal{A}$. Now, we
consider the regularized equation
\begin{align}\label{equivalenteqn}
\mathcal{A}^2\varphi=(\mathcal{H}^2+\mathcal{H}\mathcal{B}+\mathcal{B}\mathcal{H}+\mathcal{B}^2)\varphi=\mathcal{A}w,
\end{align}
which is equivalent to \eqref{Matrixequation} since $\mathcal{A}$ is invertible.

\begin{theorem}\label{Ebounded}
The operator $E^\sigma: H^{p}[0,2\pi]\rightarrow H^{p}[0,2\pi]$ is bounded.
\end{theorem}

\begin{proof}
Recalling $E^\sigma\varphi=\kappa_\sigma^2a(t)\varphi$ for
$\sigma=\mathfrak{p,s}$, where $a(t)=|z'(t)|^2$ and  is analytic, we may assume
\[
a(t)=\sum_{m=-\infty}^\infty\hat{a}_m \mathrm{e}^{\mathrm{i}mt},
\]
where $\hat{a}_m$ are the Fourier coefficients of $a$. The analyticity of $a$
implies that
\[
\sup_{m\in\mathbb{Z}}|m|^l|\hat{a}_m|<\infty\quad\forall\,l\geq 0,
\]
which, together with  \cite[Corollary 8.8]{Kress2014-book}, gives
\[
\|E^\sigma\varphi\|_{p}\leq c_1\sum_{m=-\infty}^\infty|\hat{a}_m||m|^k\|\varphi\|_p \leq c_1\sum_{m=-\infty}^\infty|\hat{a}_m|\frac{(1+m^2)^{k/2+1}}{1+m^2}\|\varphi\|_p\leq c_2\|\varphi\|_p,
\]
where the first inequality holds for all $k\geq p$.	
\end{proof}

\begin{theorem}\label{Hbounded}
The operator $\mathcal{H}: H^p[0,2\pi]^2\rightarrow H^p_*[0,2\pi]^2$ is bounded.
\end{theorem}

\begin{proof}	
For the trigonometric basis functions $f_m(t):=\mathrm{e}^{\mathrm{i}mt},
\forall m\in\mathbb{Z}$, noting
\begin{align*}
H_1f_m&=\zeta_mf_m, \qquad\zeta_m=
\begin{cases}
\mathrm{i}\,{\rm sign}(m), &m\neq0,\\
\mathrm{i},      &m=0,
\end{cases}%\label{H_1f_m}
\\
H_2f_m&=\xi_mf_m, \qquad\xi_m=
\begin{cases}
\frac{\mathrm{i}}{4}\big(\frac{1}{|m-1|}-\frac{1}{|m+1|}\big), &m=\pm2,\pm3,\cdots,\\
-\frac{\mathrm{i}}{8}{\rm sign}(m), &m=\pm1,\\
\mathrm{i},      &m=0,
\end{cases}%\label{H_2f_m}
\end{align*}
we observe that the integral operators $H_1: H^p[0,2\pi]\rightarrow
H^{p}[0,2\pi]$ and $H_2: H^p[0,2\pi]\rightarrow H^{p+2}[0,2\pi]$ are bounded for
arbitrary $p\geq0$. Then,  $\forall\varphi=(\varphi_1,\varphi_2)^\top\in
H^p[0,2\pi]^2$, using $H_1H_1+I=0$ and \eqref{norms}, we have
\begin{align*}
\|\mathcal{H}_1\varphi\|_{p,*}&=\|
(H_1\varphi_2-\varphi_1,H_1\varphi_1+\varphi_2)^\top\|_{p,*}\\ &=\|H_1\varphi_2-\varphi_1\|_p+\|H_1(H_1\varphi_1+\varphi_2)-(H_1\varphi_2-\varphi_1)\|_{p+2}\\
&\qquad
+\|H_1\varphi_1+\varphi_2\|_p+\|H_1(H_1\varphi_2-\varphi_1)+H_1\varphi_1
+\varphi_2\|_{p+2}\\
&=\|H_1\varphi_2-\varphi_1\|_p+\|H_1\varphi_1+\varphi_2\|_p\leq
C_1\|\varphi\|_p.
\end{align*}
By Theorem \ref{Ebounded}, we get
\begin{align*}
\|\mathcal{H}_2\varphi\|_{p,*}
&\leq C_2 \|\mathcal{H}_2\varphi\|_{p+2}=C_2 \|E^{\mathfrak
p}H_2\varphi_1\|_{p+2}+C_2 \|E^{\mathfrak s}H_2\varphi_2\|_{p+2}\\
&\leq C_3(\|\varphi_1\|_p+\|\varphi_2\|_p)=C_3\|\varphi\|_p,
\end{align*}
where $C_1, C_2, C_3$ are positive constants. Combining the above estimates
shows that $\mathcal{H}=\mathcal{H}_1+\mathcal{H}_2: H^p[0,2\pi]^2\rightarrow
H^p_*[0,2\pi]^2$ is bounded. 
\end{proof}

\begin{theorem}\label{compact}
The operator $\mathcal{B}: H^p[0,2\pi]^2\rightarrow H^{p+2}[0,2\pi]^2$ is compact.
\end{theorem}

\begin{proof}	
First we show that $\mathcal{B}_1$ and $\mathcal{B}_2$ are compact. Noting
\begin{align}\label{funder}
k_1^\sigma(t,t)=\partial_t k_1^\sigma(t,t)=0,\qquad
\widetilde{h}_2^\sigma(t,t)=\partial_t\widetilde{h}_2^\sigma(t,t)=0, \quad
\sigma=\mathfrak{p,s},
\end{align}
and using \cite[Theorems 12.15, 13.20]{Kress2014-book}, we get that
$K^\sigma_1,H^\sigma_2: H^p[0,2\pi]\rightarrow H^{p+3}[0,2\pi]$ are bounded for
arbitrary $p\geq0$. Thus $\mathcal{B}_1: H^p[0,2\pi]^2\rightarrow
H^{p+3}[0,2\pi]^2$ is bounded and consequently is compact from $H^p[0,2\pi]^2$
into $H^{p+2}[0,2\pi]^2$. Since the kernel functions $k_2$ and $\widetilde{h}_3$
are analytic, it follows from \cite[Theorem A.45]{Kirsch2011-book} and
\cite[Theorem 8.13]{Kress2014-book} that the operators
$K^\sigma_2,\widetilde{H}^\sigma_3: H^p[0,2\pi]\rightarrow H^{p+r}[0,2\pi]$ are
bounded for all integer $r\geq0$ and arbitrary $p\geq0$. Then the operator
$\mathcal{B}_2: H^p[0,2\pi]^2\rightarrow H^{p+r}[0,2\pi]^2$ is bounded for all
integer $r\geq0$ and arbitrary $p\geq0$. In particular, the operator
$\mathcal{B}_2: H^p[0,2\pi]^2\rightarrow H^{p+3}[0,2\pi]^2, \forall p\geq0$ is
bounded and consequently is compact from $H^p[0,2\pi]^2$ into
$H^{p+2}[0,2\pi]^2$. 

Next is to show the compactness of $\mathcal{B}_3$. It suffices to show that
$\widetilde{H}_1$ has an analytic kernel $\widetilde{h}_1$. In fact, for
$\varsigma$ sufficiently close to $t$, by using the Taylor expansions
\begin{align*}
\tan\frac{\varsigma-t}{2}&=\sum_{k=1}^\infty
a_k\Big(\frac{\varsigma-t}{2}\Big)^{2k-1},\, a_k>0,\, a_1=1,
\\
n^\perp(t)\cdot\big[z(\varsigma)-z(t)\big]&=z'_1(t)\sum_{k=1}^\infty\frac{z_1^
{(k)}(t)}{k!}(\varsigma-t)^k+z'_2(t)\sum_{k=1}^\infty\frac{z_2^{(k)}(t)}{k!}
(\varsigma-t)^k
\\&=|z'(t)|^2(\varsigma-t)\Big(1+\sum_{k=1}^\infty b_k(t)(\varsigma-t)^k\Big),
\\
|z(t)-z(\varsigma)|^2&=\Big(\sum_{k=1}^\infty\frac{z_1^{(k)}(t)}{k!}(\varsigma-t)^k\Big)^2+\Big(\sum_{k=1}^\infty\frac{z_2^{(k)}(t)}{k!}(\varsigma-t)^k\Big)^2
\\
&=|z'(t)|^2(\varsigma-t)^2\Big(1+\sum_{k=1}^\infty
d_k(t)(\varsigma-t)^k\Big)\overset{\text{def}}{=}
|z'(t)|^2(\varsigma-t)^2(1+\varTheta),
\end{align*}
we have
\begin{align*}
h_1(t,\varsigma)&=
\frac{1}{\pi}n^\perp(t)\cdot\big[z(\varsigma)-z(t)\big]\frac{\tan\frac{\varsigma-t}{2}}{|z(t)-z(\varsigma)|^2}\\
&=\frac{1}{2\pi}\Big(1+\sum_{k=1}^\infty
b_k(t)(\varsigma-t)^k\Big)(1-\varTheta+\varTheta^2-\varTheta^3+\cdots)
\sum_{k=1}^\infty a_k\Big(\frac{\varsigma-t}{2}\Big)^{2k-2}\\
&=\frac{1}{2\pi}+\sum_{k=1}^\infty c_k(t)(\varsigma-t)^k.
\end{align*}
Moreover, it can be easily verified that
\[
c_1(t)=\frac{b_1(t)-d_1(t)}{2\pi}=-\frac{z_1'(t)z_1''(t)+z_2'(t)z_2''(t)}{4\pi|z'(t)|^2}\not\equiv0.
\] 
Using the Taylor expansions for sine and cosine functions, we deduce that
the kernel function has the expansion
\[
\widetilde{h}_1(t,\varsigma)=\cot\frac{\varsigma-t}{2}\sum_{k=1}^\infty
c_k(t)(\varsigma-t)^k=\sum_{k=1}^\infty e_k(t)(\varsigma-t)^k,\quad
e_1(t)=2c_1(t),
\]
which implies that $\widetilde h_1$ is analytic and completes the proof. 
\end{proof} 

By \cite[Theorem 8.24]{Kress2014-book}, the operator $S_0:
H^p[0,2\pi]\rightarrow H^{p+1}[0,2\pi]$, defined by
\begin{align*}
(S_0\psi)(t)=\int_{0}^{2\pi}\ln\Big(4\sin^2\frac{t-\varsigma}{2}\Big)\psi(\varsigma)\,\mathrm{d}\varsigma+\sqrt{2}\mathrm{i}\int_{0}^{2\pi}\psi(\varsigma)\,\mathrm{d}\varsigma\overset{\text{def}}{=}(\widetilde{S}_0\psi)(t)+M_0,
\end{align*}
is bounded and has a bounded inverse for all $p\geq0$. We denote the operator $\mathcal{S}_0$ and $\mathcal{E}$ by
\begin{align*}
\mathcal{S}_0=\left[                  
\begin{array}{cc}
S_0   & 0 \\ 
0        & S_0 
\end{array}
\right], \quad 
\mathcal{E}=\frac{1}{8\pi^2}\left[                  
\begin{array}{cc}
E^{\mathfrak{p}}+E^{\mathfrak{s}}  & 0 \\ 
0        & E^{\mathfrak{p}}+E^{\mathfrak{s}} 
\end{array}
\right].
\end{align*}
Clearly, $\mathcal{S}_0\mathcal{S}_0$ is a isomorphism from $H^p[0,2\pi]^2$ to $H^{p+2}[0,2\pi]^2$ and $\mathcal{E}$ is a isomorphism from $H^p[0,2\pi]^2$ to $H^p[0,2\pi]^2$ since $a(t)$ is analytic and $a(t)\ne 0$.

\begin{theorem}\label{isomorphism}
For any function $\varphi\in H^p[0,2\pi]^2$,
the operator $\mathcal{H}^2: H^p[0,2\pi]^2\rightarrow H^{p+2}[0,2\pi]^2$ can be
expressed as
\begin{align*}
\mathcal{H}^2\varphi=(\mathcal{E}\mathcal{S}_0\mathcal{S}_0+\mathcal{J})\varphi,
\end{align*}
where $\mathcal{J}$ is a compact operator from $H^p[0,2\pi]^2$ into $H^{p+2}[0,2\pi]^2$.
\end{theorem}

\begin{proof}
It follows from a straightforward calculation that
\begin{align*}
\mathcal{H}^2&=\left[                  
\begin{array}{cc}
-I   & H_1+E^{\mathfrak s}H_2 \\ 
H_1+E^{\mathfrak p}H_2   & I 
\end{array}
\right]\left[                  
\begin{array}{cc}
-I   & H_1+E^{\mathfrak s}H_2 \\ 
H_1+E^{\mathfrak p}H_2   & I 
\end{array}
\right]\\
&=\left[                  
\begin{array}{cc}
H_1E^{\mathfrak p}H_2+E^{\mathfrak s}H_2H_1   & 0 \\ 
0   & H_1E^{\mathfrak s}H_2+E^{\mathfrak p}H_2H_1 
\end{array}
\right]+\mathcal{J}_1,
\end{align*}
where
\begin{align*}
\mathcal{J}_1&=\left[                  
\begin{array}{cc}
E^{\mathfrak s}J_1^{\mathfrak p}H_2   & 0 \\ 
0   & E^{\mathfrak p}J_1^{\mathfrak s}H_2
\end{array}
\right]+\left[                  
\begin{array}{cc}
E^{\mathfrak s}E^{\mathfrak p}H_2H_2   & 0 \\ 
0   & E^{\mathfrak p}E^{\mathfrak s}H_2H_2
\end{array}
\right],\\
J_1^{\sigma}\psi&=H_2E^{\sigma}\psi-E^{\sigma}H_2\psi=\frac{\kappa_{\sigma}^2}{4\pi}\int_{0}^{2\pi}\ln\Big(4\sin^2\frac{t-\varsigma}{2}\Big)\sin(t-\varsigma)\Big(|z'(\varsigma)|^2-|z'(t)|^2\Big)\psi(\varsigma)\,\mathrm{d}\varsigma\\
&+\frac{\mathrm{i}\kappa_{\sigma}^2}{2\pi}\int_{0}^{2\pi}
\Big(|z'(\varsigma)|^2-|z'(t)|^2\Big)\psi(\varsigma)\,\mathrm{d}\varsigma,
\end{align*}
and $\mathcal{J}_1$ is bounded from $H^p[0,2\pi]^2$ to $H^{p+4}[0,2\pi]^2$ and consequently is compact from $H^p[0,2\pi]^2$ into $H^{p+2}[0,2\pi]^2$.

For the first term of $\mathcal{H}^2$, we rewrite the first diagonal element by
\begin{align*}
H_1E^{\mathfrak p}H_2+E^{\mathfrak s}H_2H_1=E^{\mathfrak p}H_1H_2+E^{\mathfrak s}H_2H_1+J_2^{\mathfrak p},
\end{align*}
where $J_2^{\mathfrak p}\psi=H_1E^{\mathfrak p}H_2\psi-E^{\mathfrak
p}H_1H_2\psi=\widetilde{J}_2^{\mathfrak p}H_2\psi$ and
\begin{align*}
(\widetilde{J}_2^{\mathfrak p}\psi)(t)&=\frac{\kappa_{\mathfrak p}^2}{2\pi}\int_{0}^{2\pi}\cot\frac{t-\varsigma}{2}\Big(|z'(\varsigma)|^2-|z'(t)|^2\Big)\psi(\varsigma)\,\mathrm{d}\varsigma+\frac{\mathrm{i}\kappa_{\mathfrak p}^2}{2\pi}\int_{0}^{2\pi}\Big(|z'(\varsigma)|^2-|z'(t)|^2\Big)\psi(\varsigma)\,\mathrm{d}\varsigma.
\end{align*}
The operator $J_2^{\mathfrak p}$ is compact from $H^p[0,2\pi]^2$ into $H^{p+2}[0,2\pi]^2$, since the operator $\widetilde{J}_2^{\mathfrak p}$ has an analytic kernel and $H_2$ is bounded. In addition, for $\psi\in H^p[0,2\pi]$, we have
\begin{align*}
(H_1H_2\psi)(t)&=\frac{1}{2\pi}\int_{0}^{2\pi}\cot\frac{t-\varsigma}{2}(H_2\psi)(\varsigma)\,\mathrm{d}\varsigma+\frac{\mathrm{i}}{2\pi}\int_{0}^{2\pi}(H_2\psi)(\varsigma)\,\mathrm{d}\varsigma\\
&=\frac{1}{2\pi}\int_{0}^{2\pi}\ln\Big(4\sin^2\frac{t-\varsigma}{2}\Big)\frac{\mathrm{d}}{\mathrm{d}\varsigma}(H_2\psi)(\varsigma)\,\mathrm{d}\varsigma+\frac{\mathrm{i}}{2\pi}\xi_0\hat{\psi}_02\pi\\
&=\frac{1}{2\pi}\frac{1}{4\pi}\int_{0}^{2\pi}\ln\Big(4\sin^2\frac{t-\varsigma}{2}\Big)\left\{\frac{\mathrm{d}}{\mathrm{d}\varsigma}\int_0^{2\pi}\ln\Big(4\sin^2\frac{\varsigma-s}{2}\Big)\sin(\varsigma-s)\psi(s)\,\mathrm{d}s\right\}\,\mathrm{d}\varsigma-\hat{\psi}_0\\
&=\frac{1}{8\pi^2}\int_{0}^{2\pi}\ln\Big(4\sin^2\frac{t-\varsigma}{2}\Big)\int_{0}^{2\pi}\ln\Big(4\sin^2\frac{\varsigma-s}{2}\Big)\psi(s)\,\mathrm{d}s\,\mathrm{d}\varsigma-\hat{\psi}_0+\frac{1}{8\pi^2}(J_3\psi+J_4\psi)(t)\\
&=\frac{1}{8\pi^2}(S_0S_0\psi)(t)+\frac{1}{8\pi^2}(J_3\psi)(t)+\frac{1}{8\pi^2}(J_4\psi)(t),
\end{align*}
where
\begin{align*}
(J_3\psi)(t)&=\int_{0}^{2\pi}\ln\Big(4\sin^2\frac{t-\varsigma}{2}\Big)\int_{0}^{2\pi}\cot\frac{\varsigma-s}{2}\sin(\varsigma-s)\psi(s)\,\mathrm{d}s\overset{\text{def}}{=}(\widetilde{S}_0\widetilde{J}_3\psi)(t),\\
(J_4\psi)(t)&=\int_{0}^{2\pi}\ln\Big(4\sin^2\frac{t-\varsigma}{2}\Big)\int_{0}^{2\pi}\ln\Big(4\sin^2\frac{\varsigma-s}{2}\Big)(\cos(\varsigma-s)-1)\psi(s)\,\mathrm{d}s\overset{\text{def}}{=}(\widetilde{S}_0\widetilde{J}_4\psi)(t),
\end{align*}
are compact from $H^p[0,2\pi]^2$ into $H^{p+2}[0,2\pi]^2$, since the operator
$\widetilde{J}_3$ has an analytic kernel and the operator $\widetilde{S}_0:
H^p[0,2\pi]^2\rightarrow H^{p+1}[0,2\pi]^2$ and the operator $J_4:
H^p[0,2\pi]^2\rightarrow H^{p+3}[0,2\pi]^2$ are bounded. Clearly it also holds
that $H_2H_1=H_1H_2$ . 

Similarly, we can analyze the second diagonal element as the first one. Therefore, we obtain the assertion of the theorem by defining the operator
\begin{align*}
\mathcal{J}=\mathcal{J}_1+\mathcal{J}_2+\mathcal{J}_3+\mathcal{J}_4,
\end{align*}
where
\[
\mathcal{J}_2=\left[                  
\begin{array}{cc}
J_2^{\mathfrak p}  & 0 \\ 
0   & J_2^{\mathfrak s}
\end{array}
\right],\quad 
\mathcal{J}_3=\frac{\kappa_{\mathfrak p}^2+\kappa_{\mathfrak s}^2}{8\pi^2}|z'|^2\left[                  
\begin{array}{cc}
J_3  & 0 \\ 
0   & J_3
\end{array}
\right], \quad 
\mathcal{J}_4=\frac{\kappa_{\mathfrak p}^2+\kappa_{\mathfrak s}^2}{8\pi^2}|z'|^2\left[                  
\begin{array}{cc}
J_4  & 0 \\ 
0   & J_4
\end{array}
\right].
\]
\end{proof}

%%%%%%%%%%%%%%%%%%%%%%%%%%%%%%%%%%%%%%%%%%%%%%%%%%%%%%%%%
\section{Trigonometric collocation method}

Consider the following equivalent formulation of the operator equation
\eqref{equivalenteqn}: 
\begin{align}\label{Matrixequation2}
\mathcal{S}_0\mathcal{S}_0\varphi+\mathcal{E}^{-1}(\mathcal{J}+\mathcal{K}
)\varphi=\mathcal{E}^{-1}\mathcal{A}w,
\end{align}
where $\mathcal{K}\overset{\text{def}}{=}\mathcal{H}\mathcal{B}+\mathcal{B}\mathcal{H}+\mathcal{B}^2$
is a compact operator from $H^p[0,2\pi]^2$ into $H^{p+2}[0,2\pi]^2$ by Theorems
\ref{Hbounded} and \ref{compact}. In this section, we examine the convergence
of the semi- and full-discretization of \eqref{Matrixequation2} by using the
trigonometric collocation method. 

\subsection{Semi-discretization}

Let $X_n$ be the space of trigonometric polynomials of degree less than or equal to $n$ of the form
\begin{align}\label{trigopoly}
\varphi(t)=\sum_{m=0}^n\alpha_m\cos mt+\sum_{m=1}^{n-1}\beta_m\sin mt.
\end{align}
Denote by $P_n: H^p[0,2\pi]\rightarrow X_n$ the interpolation operator, which
maps $2\pi$-periodic scalar function $g$ into a unique trigonometric polynomial
$P_ng$ at the equidistant interpolation points $\varsigma_j^{(n)}:=\pi j/n$,
$j=0,\cdots,2n-1$, i.e., $(P_ng)(\varsigma_j^{(n)})=g(\varsigma_j^{(n)})$. Then,
$P_n$ is a bounded linear operator.

Let $X_n^2=\{\varphi=(\varphi_1,\varphi_2)^\top: \varphi_1\in X_n,
\varphi_2\in X_n\}$ and define the interpolation operator $\mathcal{P}_n:
H^p[0,2\pi]^2\rightarrow X_n^2$ by
$$
\mathcal{P}_ng=(P_ng_1,P_ng_2)^\top,\quad\forall\,g=(g_1,g_2)\in H^p[0,2\pi]^2.
$$
Clearly, $X_n^2$ is unisolvent with respect to the points
$\{\varsigma_j^{(n)}\}_{j=0}^{2n-1}$. Moreover, we have from \cite[Theorem
11.8]{Kress2014-book} that 
\begin{align}\label{interpolationerror}
\|\mathcal{P}_ng-g\|_{q}\leq\frac{C}{n^{p-q}}\|g\|_{p},\qquad 0\leq q\leq p, \quad \frac{1}{2}<p
\end{align}
for all $g\in H^p[0,2\pi]^2$ and some constant $C$ depending on $p$ and $q$.

Now we approximate the solution $\varphi=(\varphi_1,\varphi_2)^\top$ by a
trigonometric polynomial $\varphi^n=(\varphi^n_1,\varphi^n_2)^\top\in X_n^2$,
which is required to satisfy the projected equation
\begin{align}\label{semicollocation}
\mathcal{S}_0\mathcal{S}_0\varphi^n +
\mathcal{P}_n[\mathcal{E}^{-1}(\mathcal{J}+\mathcal{K})]\varphi^n =
\mathcal{P}_n(\mathcal{E}^{-1}\mathcal{A})w,
\end{align}
where $\varphi^n$
satisfies $\mathcal{P}_n(\mathcal{S}_0\mathcal{S}_0)\varphi^n=\mathcal{S}
_0\mathcal{S} _0\varphi^n$.

\begin{theorem}
For sufficiently large $n$, the approximate equation \eqref{semicollocation} is
uniquely solvable and the solution satisfies the error estimate
\begin{align*}%\label{errorestimate1}
\|\varphi^n-\varphi\|_p\leq
L\|\mathcal{P}_n\mathcal{S}_0\mathcal{S}_0\varphi-\mathcal{S}_0\mathcal{S}
_0\varphi\|_{p+2},
\end{align*}
where $L$ is a positive constant depending on 
$\mathcal{E}^{-1}\mathcal{J}$, $\mathcal{E}^{-1}\mathcal{K}$ and
$\mathcal{S}_0\mathcal{S}_0$.
\end{theorem}

\begin{proof}
From the proofs of Theorems \ref{compact} and \ref{isomorphism}, we know
that the operators $\mathcal{J}, \mathcal{K}: H^p[0,2\pi]^2\rightarrow
H^{p+3}[0,2\pi]^2, \forall p\geq0$ are bounded. With the aid of
\eqref{interpolationerror}, we deduce that
\begin{align*}
\|\mathcal{P}_n[\mathcal{E}^{-1}(\mathcal{J}+\mathcal{K})]\varphi-\mathcal{E}^{-1}(\mathcal{J}+\mathcal{K})\varphi\|_{p+2}\leq\frac{c_1}{n}\|\mathcal{E}^{-1}(\mathcal{J}+\mathcal{K})\varphi\|_{p+3}\leq\frac{c_2}{n}\|\varphi\|_{p}
\end{align*}
for all $p\geq0$ and some constants $c_1$ and $c_2$, which implies
\begin{align*}%\label{estimate1}
\|\mathcal{P}_n[\mathcal{E}^{-1}(\mathcal{J}+\mathcal{K})]-\mathcal{E}^{-1}
(\mathcal{J}+\mathcal{K})\|_{p+2}\rightarrow 0\quad \text{as} \quad
n\rightarrow\infty.
\end{align*}
The proof is completed by noting \cite[Theorem 13.12]{Kress2014-book}.
\end{proof}

The above theorem implies that the semi-discrete collocation method given by
\eqref{semicollocation} converges in $H^p[0,2\pi]^2$ for each $p\geq 0$.

\subsection{Full-discretization}

Denote the Lagrange basis by
$$
\mathfrak{L}_j(t)=\frac{1}{2n}\left\{1+2\sum_{k=1}^{n-1}\cos k(t-\varsigma_j^{(n)})+\cos n(t-\varsigma_j^{(n)})\right\}, \quad j=0,1,\cdots,2n-1.
$$
Instead of \eqref{semicollocation}, we find an approximate solution
$\widetilde\varphi^n\in X_n^2$ given by
$$\widetilde\varphi^n(t)=\big(\widetilde\varphi^n_1(t),
\widetilde\varphi^n_2(t)\big)^\top=\Big(\sum_{j=0}^{2n-1}
\widetilde\varphi^n_1(\varsigma_j^{(n)})\mathfrak{L}_j(t) ,
\sum_{j=0}^{2n-1}\widetilde\varphi^n_2(\varsigma_j^{(n)})
\mathfrak{L}_j(t)\Big)^\top,
$$ 
which is required to satisfy 
\begin{align}\label{fullcollocation}
\mathcal{S}_{0,n}\mathcal{S}_{0,n}\widetilde\varphi^n+\mathcal{P}_n[\mathcal{E}_n^{-1}(\mathcal{J}_n+\mathcal{K}_n)]\widetilde\varphi^n=\mathcal{P}_n(\mathcal{E}_n^{-1}\mathcal{A}_n)w, 
\end{align}
where $\displaystyle\mathcal{A}_n=\mathcal{H}_n+\mathcal{B}_n=\sum_{j=1}^{2}\mathcal{H}_{j,n}+\sum_{j=1}^{3}\mathcal{B}_{j,n}$, $\displaystyle\mathcal{J}_n=\sum_{j=1}^{4}\mathcal{J}_{j,n}$, $\mathcal{K}_n=\mathcal{H}_n\mathcal{B}_n+\mathcal{B}_n\mathcal{H}_n+\mathcal{B}_n\mathcal{B}_n$, and the quadrature operators are described by $\mathcal{S}_{0,n}=\mathcal{S}_{0}\mathcal{P}_n$, $\mathcal{H}_{1,n}=\mathcal{H}_{1}\mathcal{P}_n$, $\mathcal{H}_{2,n}=\mathcal{H}_{2}\mathcal{P}_n$,
\begin{align*}
%%%%%%%%%%%%%%%%%%%%%
(\mathcal{B}_{1,n}\chi)(t)&=\int_0^{2\pi}\ln\Big(4\sin^2\frac{t-\varsigma}{2}\Big)\mathcal{P}_n \Bigg\{\left[\begin{array}{cc}
k_1^{\mathfrak p}(t,\cdot)& \widetilde{h}_2^{\mathfrak s}(t,\cdot)\\ 
\widetilde{h}_2^{\mathfrak p}(t,\cdot)& -k_1^{\mathfrak s}(t,\cdot;)
\end{array}
\right]\chi\Bigg\}(\varsigma)\,\mathrm{d}\varsigma,\\
%%%%%%%%%%%%%%%%%%%%%%%%%%%%%%%%%%%%%%%%%%%%%%%%%%%%
(\mathcal{B}_{2,n}\chi+\mathcal{B}_{3,n}\chi)(t)&=\int_0^{2\pi}\mathcal{P}_n \Bigg\{\left[\begin{array}{cc}
k_2^{\mathfrak p}(t,\cdot)& \widetilde{h}_3^{\mathfrak s}(t,\cdot)+\widetilde{h}_1(t,\cdot) \\ 
\widetilde{h}_3^{\mathfrak p}(t,\cdot)+\widetilde{h}_1(t,\cdot)& -k_2^{\mathfrak s}(t,\cdot)
\end{array}
\right]\chi\Bigg\}(\varsigma)\,\mathrm{d}\varsigma,
\end{align*}
and define $E_n^{\sigma}\psi=\kappa_{\sigma}^2|z'|^2\psi=E^{\sigma}\psi$, $H_{2,n}=H_2P_n$, $\widetilde{S}_{0,n}=\widetilde{S}_{0}P_n$,
\begin{align*}
\mathcal{E}_n&=\frac{1}{8\pi^2}\left[             \begin{array}{cc}
E_n^{\mathfrak{p}}+E_n^{\mathfrak{s}}  & 0 \\ 
0        & E_n^{\mathfrak{p}}+E_n^{\mathfrak{s}} 
\end{array}
\right],\\
%%%%%%%%%%%%%%%%%%%
\mathcal{J}_{1,n}&=\left[                  
\begin{array}{cc}
E_n^{\mathfrak s}J_{1,n}^{\mathfrak p}H_{2,n}   & 0 \\ 
0   & E_n^{\mathfrak p}J_{1,n}^{\mathfrak s}H_{2,n}
\end{array}
\right]+\left[                  
\begin{array}{cc}
E_n^{\mathfrak s}E_n^{\mathfrak p}H_{2,n}H_{2,n}   & 0 \\ 
0   & E_n^{\mathfrak p}E_n^{\mathfrak s}H_{2,n}H_{2,n}
\end{array}
\right], \\
%%%%%%%%%%%%%%%%%%%%%%%%%%%%%%%%%%%%%%%%
\mathcal{J}_{2,n}&=\left[                  
\begin{array}{cc}
\widetilde{J}_{2,n}^{\mathfrak p}H_{2,n}  & 0 \\ 
0   & \widetilde{J}_{2,n}^{\mathfrak s}H_{2,n}
\end{array}
\right],\quad 
\mathcal{J}_{3,n}=\frac{\kappa_{\mathfrak p}^2+\kappa_{\mathfrak s}^2}{8\pi^2}|z'|^2\left[                  
\begin{array}{cc}
\widetilde{S}_{0,n}\widetilde{J}_{3,n}  & 0 \\ 
0   & \widetilde{S}_{0,n}\widetilde{J}_{3,n}
\end{array}
\right], \\
%%%%%%%%%%%%%%%%%%%%%%%%%%%%%%%%%%%%%%%%%%%%%%
\mathcal{J}_{4,n}&=\frac{\kappa_{\mathfrak p}^2+\kappa_{\mathfrak s}^2}{8\pi^2}|z'|^2\left[                  
\begin{array}{cc}
\widetilde{S}_{0,n}\widetilde{J}_{4,n}  & 0 \\ 
0   & \widetilde{S}_{0,n}\widetilde{J}_{4,n}
\end{array}
\right],
\end{align*}
where
\begin{align*}
(J_{1,n}^{\sigma}\psi)(t)&=\frac{\kappa_{\sigma}^2}{4\pi}\int_{0}^{2\pi}\ln\Big(4\sin^2\frac{t-\varsigma}{2}\Big)P_n\left\{\sin(t-\cdot)\Big(|z'(\cdot)|^2-|z'(t)|^2\Big)\psi\right\}(\varsigma)\,\mathrm{d}\varsigma\\
&\qquad +\frac{\mathrm{i}\kappa_{\sigma}^2}{2\pi}\int_{0}^{2\pi}P_n\left\{
\Big(|z'(\cdot)|^2-|z'(t)|^2\Big)\psi\right\}(\varsigma)\,\mathrm{d}\varsigma,\\
%%%%%%%%%%%%%%%%%%%%%%%%%%%%%%%%%%%%%%%%%%
(\widetilde{J}_{2,n}^{\sigma}\psi)(t)&=\frac{\kappa_\sigma^2}{2\pi}\int_{0}^{2\pi}P_n\left\{(\cot\frac{t-\cdot}{2}+\mathrm{i})\Big(|z'(\cdot)|^2-|z'(t)|^2\Big)\psi\right\}(\varsigma)\,\mathrm{d}\varsigma,\\
%%%%%%%%%%%%%%%%%%%%%%%%%%%%%%%%%%%%%%%%%%%%%
(\widetilde{J}_{3,n}\psi)(t)&=\int_{0}^{2\pi}P_n\left\{\cot\frac{t-\cdot}{2}\sin(t-\cdot)\psi\right\}(\varsigma)\,\mathrm{d}\varsigma,\\
%%%%%%%%%%%%%%%%%%%%%%%%%%%%%%%%%%%%%%%%%%%%%
(\widetilde{J}_{4,n}\psi)(t)&=\int_{0}^{2\pi}\ln\Big(4\sin^2\frac{t-\varsigma}{2}\Big)P_n\Big\{(\cos(t-\cdot)-1)\psi\Big\}(\varsigma)\,\mathrm{d}\varsigma.
%\\
%%%%%%%%%%%%%%%%%%%%%%%%%%%%%%%%%%%%%%%%%%%%%%%%%%
\end{align*}
Clearly, $\mathcal{S}_{0,n}\mathcal{S}_{0,n}\widetilde\varphi^n=\mathcal{S}_0\mathcal{S}_0\widetilde\varphi^n$, $\mathcal{H}_n\widetilde\varphi^n=\mathcal{H}\widetilde\varphi^n$ for $\widetilde\varphi^n\in X_n^2$.

In the following, we show the convergence of the full-discrete collocation
method \eqref{fullcollocation}. To this end, we rewrite the function
$\mathcal{B}\varphi$ in form of 
\begin{align*}
(\mathcal{B}\varphi)(t)=\int_0^{2\pi}\ln\Big(4\sin^2\frac{t-\varsigma}{2}
\Big)M(t,\varsigma)\varphi(\varsigma)\,\mathrm{d}\varsigma+\int_0^{2\pi}N(t,
\varsigma)\varphi(\varsigma)\,\mathrm{d}\varsigma,
\end{align*}
where
$$
M(t,\varsigma)=\left[\begin{array}{cc}
m_1(t,\varsigma)& m_2(t,\varsigma)\\ 
m_3(t,\varsigma)& m_4(t,\varsigma)
\end{array}
\right], \qquad N(t,\varsigma)=\left[\begin{array}{cc}
n_1(t,\varsigma)& n_2(t,\varsigma)\\ 
n_3(t,\varsigma)& n_4(t,\varsigma)
\end{array}
\right].
$$
Noting \eqref{funder} and the analyticity of
$k_j^\sigma(t,\varsigma),j=1,2$, $\widetilde{h}_2^\sigma(t,\varsigma)$,
$\widetilde{h}_3^\sigma(t,\varsigma)$ and the kernel of $\widetilde{H}_1$, we
conclude that $m_j(t,t)=\partial_t m_j(t,t)=0, j=1,2,3,4$ and $m_j, n_j$
are analytic. Recall that the full discretization of $\mathcal{B}$ is
$\mathcal{B}_n=\mathcal{B}_{1,n}+\mathcal{B}_{2,n}+\mathcal{B}_{3,n}$.

\begin{theorem}
\label{Bn-estimate}
Assume that $0\leq q\leq p$ and $p>1/2$. Then for the quadrature operator
$\mathcal{B}_n$, the following estimates hold:
\begin{align}
\label{Bestimate}
\|\mathcal{B}_n\varphi-\mathcal{B}\varphi\|_{q+2}\leq C\frac{1}{n^{p+1-q}}\|\varphi\|_{p},\quad \|\mathcal{B}_n\chi-\mathcal{B}\chi\|_{q+2}\leq \widetilde{C}\frac{1}{n^{p-q}}\|\chi\|_{p}
\end{align}
for all trigonometric polynomials $\varphi\in X_n^2$ and all $\chi\in H^p[0,2\pi]^2$, and some constants $C, \widetilde{C}$ depending on $p$ and $q$.
\end{theorem}

\begin{proof}
For the derivative $\frac{\rm d}{{\rm d}t}(\mathcal{B}\varphi)$, it can be written in form of
\begin{align*}
(\mathcal{B}'\varphi)(t)\overset{\text{def}}{=}\frac{\rm d}{{\rm
d}t}(\mathcal{B}\varphi)(t)=\int_0^{2\pi}\ln\Big(4\sin^2\frac{t-\varsigma}{2}
\Big)\widetilde M(t,\varsigma)
\varphi(\varsigma)\,\mathrm{d}\varsigma+\int_0^{2\pi}\widetilde N(t,\varsigma)
\varphi(\varsigma)\,\mathrm{d}\varsigma,
\end{align*}
where
$$
\widetilde M(t,\varsigma)=\left[\begin{array}{cc}
\partial_tm_1(t,\varsigma)& \partial_tm_2(t,\varsigma)\\ 
\partial_tm_3(t,\varsigma)& \partial_tm_4(t,\varsigma)
\end{array}
\right],\quad \widetilde N(t,\varsigma) =M(t,\varsigma)\cot\frac{t-\varsigma}{2}+\left[\begin{array}{cc}
\partial_tn_1(t,\varsigma)& \partial_tn_2(t,\varsigma)\\ 
\partial_tn_3(t,\varsigma)& \partial_tn_4(t,\varsigma)
\end{array}
\right].
$$
We denote the full discretization of $\mathcal{B}'$ via interpolatory quadrature by
\begin{align*}
(\mathcal{B}'_n\varphi)(t)=\int_0^{2\pi}\ln\Big(4\sin^2\frac{t-\varsigma}{2}\Big)\mathcal{P}_n\Big\{\widetilde M(t,\cdot) \varphi\Big\}(\varsigma)\,\mathrm{d}\varsigma+\int_0^{2\pi}\mathcal{P}_n\Big\{\widetilde N(t,\cdot) \varphi\Big\}(\varsigma)\,\mathrm{d}\varsigma.
\end{align*}

For $p>1/2$ and the integer $q$ satisfying $0\leq q\leq p$, from $\widetilde M(t,t)=0$ together with the analyticity of the elements in $\widetilde M(t,\varsigma)$ and $\widetilde N(t,\varsigma)$, using \cite[Lemma 13.21 and Theorem 12.18]{Kress2014-book}, we have
$$
\|\mathcal{B}'_n\varphi-\mathcal{B}'\varphi\|_{q+1}\leq C_1\frac{1}{n^{p+1-q}}\|\varphi\|_{p},\quad \|\mathcal{B}'_n\chi-\mathcal{B}'\chi\|_{q+1}\leq \widetilde{C}_1\frac{1}{n^{p-q}}\|\chi\|_{p}
$$
for all trigonometric polynomials $\varphi\in X_n^2$ and some constants $C_1$, $\widetilde{C}_1$ depending on $p$ and $q$.
Noting $\mathcal{B}'_n\varphi=\frac{\rm d}{{\rm d}t}(\mathcal{B}_n\varphi)$, the above equation implies 
\begin{align*}
\|\mathcal{B}_n\varphi-\mathcal{B}\varphi\|_{q+2}\leq C_2\frac{1}{n^{p+1-q}}\|\varphi\|_{p}, \quad \|\mathcal{B}_n\chi-\mathcal{B}\chi\|_{q+2}\leq \widetilde{C}_2\frac{1}{n^{p-q}}\|\chi\|_{p}
\end{align*}
for some constants $C_2$, $\widetilde{C}_2$ depending on $p$ and $q$.
It follows from \cite[Theorem 8.13]{Kress2014-book} that the above inequality
holds for arbitrary $q$ satisfying
$0\leq q\leq p$ and $p>1/2$, which completes the proof.
\end{proof}

In the following, the notation $a\lesssim b$ means $a\leq Cb$, where $C > 0$ is a constant depending on $p$.

\begin{theorem}
\label{Kn-estimate}
Assume that $p>1/2$. Then for the quadrature operator $\mathcal{K}_n$, the
following estimate holds: 
$$
\|\mathcal{P}_n[\mathcal{E}_n^{-1}\mathcal{K}_n-\mathcal{E}^{-1}\mathcal{K}]\varphi\|_{p+2}\lesssim \frac{1}{n}\|\varphi\|_{p}
$$
for all trigonometric polynomials $\varphi\in X_n^2$.
\end{theorem}

\begin{proof}
From Theorem \ref{Hbounded} and the estimate \eqref{interpolationerror}, $\forall\chi\in H^p[0,2\pi]^2$, we have
\begin{align*}%\label{Hestimate}
\|(\mathcal{H}_n-\mathcal{H})\chi\|_{q,*}=\|\mathcal{H}(\mathcal{P}_n\chi-\chi)\|_{q,*}\leq C_1\|(\mathcal{P}_n\chi-\chi)\|_{q}\leq\frac{C_2}{n^{p-q}}\|\chi\|_{p}
\end{align*} 
for $0\leq q\leq p$, $p\geq1/2$ and some constants $C_1$ depending on $q$ and $C_2$ depending on $p$ and $q$. Then, $\mathcal{H}_n$ and $\mathcal{H}_n-\mathcal{H}$ are uniformly bounded from $H^p[0,2\pi]^2$ to $H^p_*[0,2\pi]^2$ for $p\geq1/2$. Clearly, $\mathcal{E}_n^{-1}\chi=\mathcal{E}^{-1}\chi$.

For all trigonometric polynomials $\varphi\in X_n^2$ of the form \eqref{trigopoly}, using Theorem \ref{Bn-estimate} and the fact that $\mathcal{H}_n\varphi=\mathcal{H}\varphi$, we get
\begin{align*}
&\|\mathcal{H}_n\mathcal{B}_n\varphi-\mathcal{H}\mathcal{B}\varphi\|_{p+2}
\leq\|\mathcal{H}_n\mathcal{B}_n\varphi-\mathcal{H}\mathcal{B}\varphi\|_{p+2,*}
\\
& \leq 
\|\mathcal{H}_n(\mathcal{B}_n-\mathcal{B})\varphi\|_{p+2,*}+\|(\mathcal{H}
_n-\mathcal{H})(\mathcal{B}\varphi-\mathcal{P}_n\mathcal{B}\varphi)\|_{p+2,*}
+\|(\mathcal{H}_n-\mathcal{H})\mathcal{P}_n\mathcal{B}\varphi\|_{p+2,*}\\
& \lesssim \|(\mathcal{B}_n-\mathcal{B})\varphi\|_{p+2}+\|\mathcal{B}
\varphi-\mathcal{P} _n\mathcal{B}\varphi\|_{p+2}\\
& \lesssim 1/n\|\varphi\|_{p}+1/n\|\mathcal{B}\varphi\|_{p+3}
\lesssim1/n\|\varphi\|_{p}
\end{align*}
Furthermore, it follows from Theorem \ref{Bn-estimate} that
$\mathcal{B}_n$ and $\mathcal{B}_n-\mathcal{B}$ are uniformly bounded from
$H^p[0,2\pi]^2$ to $H^{p+2}[0,2\pi]^2$ for $p\geq1/2$. Thus, using
\eqref{interpolationerror}, \eqref{Bestimate} and the uniform boundedness of
$\mathcal{P}_n: H^{p+2}[0,2\pi]^2\rightarrow H^{p+2}[0,2\pi]^2$, together with
the boundedness of $\mathcal{B}: H^p[0,2\pi]^2\rightarrow H^{p+3}[0,2\pi]^2$,
we deduce
\begin{align*}
\|\mathcal{B}^2_n\varphi-\mathcal{B}^2\varphi\|_{p+2}
&\leq \|\mathcal{B}^2_n\varphi-\mathcal{B}^2\varphi\|_{p+4}\\
&\leq \|\mathcal{B}_n(\mathcal{B}_n-\mathcal{B})\varphi\|_{p+4}+\|(\mathcal{B}_n-\mathcal{B})(\mathcal{B}\varphi-\mathcal{P}_n\mathcal{B}\varphi)\|_{p+4}+\|(\mathcal{B}_n-\mathcal{B})\mathcal{P}_n\mathcal{B}\varphi\|_{p+4}\\
&\lesssim  \|(\mathcal{B}_n-\mathcal{B})\varphi\|_{p+2}+\|(\mathcal{B}\varphi-\mathcal{P}_n\mathcal{B}\varphi)\|_{p+2}+1/n\|\mathcal{P}_n\mathcal{B}\varphi\|_{p+2}\\
&\lesssim  1/n\|\varphi\|_{p}+1/n\|\mathcal{B}\varphi\|_{p+3}+1/n\|\mathcal{B}\varphi\|_{p+2}
\lesssim 1/n\|\varphi\|_{p}. 
\end{align*}
Noting $\mathcal{H}_1\varphi\in X_n^2$, we obtain 
\begin{align*}
\|\mathcal{B}_n\mathcal{H}_n\varphi-\mathcal{B}\mathcal{H}\varphi\|_{p+2}&=\|(\mathcal{B}_n-\mathcal{B})\mathcal{H}\varphi\|_{p+2}\\
&\leq  \|(\mathcal{B}_n-\mathcal{B})(\mathcal{H}\varphi-\mathcal{P}_n\mathcal{H}\varphi)\|_{p+2}+\|(\mathcal{B}_n-\mathcal{B})\mathcal{P}_n\mathcal{H}\varphi\|_{p+2}\\
&\lesssim  \|\mathcal{H}\varphi-\mathcal{P}_n\mathcal{H}\varphi\|_{p}+1/n\|\mathcal{P}_n\mathcal{H}\varphi\|_p\\
&\leq \|\mathcal{H}_1\varphi-\mathcal{P}_n\mathcal{H}_1\varphi\|_{p}+\|\mathcal{H}_2\varphi-\mathcal{P}_n\mathcal{H}_2\varphi\|_{p}+1/n\|\mathcal{P}_n\mathcal{H}\varphi\|_p\\
&\lesssim 
1/n^2\|\mathcal{H}_2\varphi\|_{p+2}+1/n\|\mathcal{H}\varphi\|_p
\lesssim 1/n\|\varphi\|_{p}.
\end{align*}
Therefore
\begin{align*}
\|\mathcal{K}_n\varphi-\mathcal{K}\varphi\|_{p+2}
\leq
\|\mathcal{H}_n\mathcal{B}_n\varphi-\mathcal{H}\mathcal{B}\varphi\|_{p+2}+\|\mathcal{B}_n\mathcal{H}_n\varphi-\mathcal{B}\mathcal{H}\varphi\|_{p+2}+\|\mathcal{B}_n^2\varphi-\mathcal{B}^2\varphi\|_{p+2}
\lesssim 1/n\|\varphi\|_{p}.
\end{align*}
The proof is completed by using the uniform boundedness of the
operator $\mathcal{E}^{-1},\mathcal{P}_n: H^{p+2}[0,2\pi]^2\rightarrow
H^{p+2}[0,2\pi]^2$. 
\end{proof}

\begin{theorem}
\label{Jn-estimate}
Assume that $p>1/2$. Then for the quadrature operator $\mathcal{J}_n$, the
following estimate holds: 
$$
\|\mathcal{P}_n[\mathcal{E}_n^{-1}\mathcal{J}_n-\mathcal{E}^{-1}\mathcal{J}]\varphi\|_{p+2}
\lesssim 
\frac{1}{n}\|\varphi\|_{p}
$$
for all trigonometric polynomials $\varphi\in X_n^2$.
\end{theorem}

\begin{proof}
For all trigonometric polynomials $\varphi\in X_n^2$, we claim that
$$
\|\mathcal{J}_n\varphi-\mathcal{J}\varphi\|_{p+2}
\lesssim 
\frac{1}{n}\|\varphi\|_{p}.
$$
In fact, analogous to the discussion in Theorem \ref{Bn-estimate}, we get
\begin{align*}
\|J_{1,n}^{\sigma}\psi-J_{1}^{\sigma}\psi\|_{q+2}
&\leq 
C\frac{1}{n^{p+1-q}}\|\psi\|_{p},\qquad 0\leq q\leq p, \quad \frac{1}{2}<p,\\
\|\widetilde{J}_{4,n}\psi-\widetilde{J}_{4}\psi\|_{q+2}
&\leq 
C\frac{1}{n^{p+1-q}}\|\psi\|_{p},\qquad 0\leq q\leq p, \quad \frac{1}{2}<p, 
\end{align*}
for all trigonometric polynomials $\psi\in X_n$ and some constant $C$ depending on $p$ and $q$. Since $E_n^{\sigma}\varphi_j=E^{\sigma}\varphi_j$, 
$H_{2,n}\varphi_j=H_2\varphi_j\in X_n$, $j=1,2$, we have
\begin{align*}
\|\mathcal{J}_{1,n}\varphi-\mathcal{J}_1\varphi\|_{p+2}
&=
\|(E_n^{\mathfrak s}J_{1,n}^{\mathfrak p}H_{2,n}-E^{\mathfrak s}J_{1}^{\mathfrak p}H_{2})\varphi_1\|_{p+2}+\|(E_n^{\mathfrak p}J_{1,n}^{\mathfrak s}H_{2,n}-E^{\mathfrak p}J_{1}^{\mathfrak s}H_{2})\varphi_2\|_{p+2}\\
&\lesssim
\|(J_{1,n}^{\mathfrak p}-J_{1}^{\mathfrak p})H_{2}\varphi_1\|_{p+2}+\|(J_{1,n}^{\mathfrak s}-J_{1}^{\mathfrak s})H_{2}\varphi_2\|_{p+2}\\
&\leq
\|(J_{1,n}^{\mathfrak p}-J_{1}^{\mathfrak p})H_{2}\varphi_1\|_{p+4}+\|(J_{1,n}^{\mathfrak s}-J_{1}^{\mathfrak s})H_{2}\varphi_2\|_{p+4}\\
&\lesssim
1/n\|H_{2}\varphi_1\|_{p+2}+1/n\|H_{2}\varphi_2\|_{p+2}
\lesssim
1/n\|\varphi\|_{p}.
\end{align*}
Since the operator $\widetilde{J}_2^{\sigma}$ has an analytic kernel, it is easy to see
$$
\|\mathcal{J}_{2,n}\varphi-\mathcal{J}_2\varphi\|_{p+2}
\lesssim
\frac{1}{n}\|\varphi\|_{p}.
$$

In addition, in terms of $\widetilde{S}_{0,n}\psi=\widetilde{S}_{0}\psi$ for $\psi\in X_n$ and the uniform boundedness of $\widetilde{S}_{0,n}$ and $\widetilde{S}_{0,n}-\widetilde{S}_{0}$ from $H^p[0,2\pi]^2$ to $H^{p+1}[0,2\pi]^2$ for $p\geq1/2$, we obtain
\begin{align*}
\|\mathcal{J}_{4,n}\varphi-\mathcal{J}_4\varphi\|_{p+2}
&=
\frac{1}{8\pi^2}\sum_{j=1}^{2}
\Big\|\Big((E_n^{\mathfrak p}+E_n^{\mathfrak s})\widetilde{S}_{0,n}\widetilde{J}_{4,n}-(E^{\mathfrak p}+E^{\mathfrak s})\widetilde{S}_{0}\widetilde{J}_{4}\Big)\varphi_j\Big\|_{p+2}\\
&\lesssim
\sum_{j=1}^{2}
\|(\widetilde{S}_{0,n}\widetilde{J}_{4,n}-\widetilde{S}_{0}\widetilde{J}_{4})\varphi_j\|_{p+2}\\
&\leq
\sum_{j=1}^{2}\left(
\|\widetilde{S}_{0,n}(\widetilde{J}_{4,n}-\widetilde{J}_{4})\varphi_j\|_{p+2}+\|(\widetilde{S}_{0,n}-\widetilde{S}_{0})(\widetilde{J}_{4}\varphi_j-P_n\widetilde{J}_{4}\varphi_j)\|_{p+2}\right)\\
&\leq
\sum_{j=1}^{2}\left(
\|\widetilde{S}_{0,n}(\widetilde{J}_{4,n}-\widetilde{J}_{4})\varphi_j\|_{p+3}+\|(\widetilde{S}_{0,n}-\widetilde{S}_{0})(\widetilde{J}_{4}\varphi_j-P_n\widetilde{J}_{4}\varphi_j)\|_{p+3}\right)\\
&\lesssim
\sum_{j=1}^{2}\left(
\|(\widetilde{J}_{4,n}-\widetilde{J}_{4})\varphi_j\|_{p+2}+\|(\widetilde{J}_{4}\varphi_j-P_n\widetilde{J}_{4}\varphi_j)\|_{p+2}\right)\\
&\lesssim
\sum_{j=1}^{2}\left(
\frac{1}{n}\|\varphi_j\|_{p+2}+\frac{1}{n}\|(\widetilde{J}_{4}\varphi_j\|_{p+3}\right)
\lesssim
\frac{1}{n}\|\varphi\|_{p}.
\end{align*}
Since the operator $\widetilde{J}_3$ has an analytic kernel, similarly we get
$$
\|\mathcal{J}_{3,n}\varphi-\mathcal{J}_3\varphi\|_{p+2}
\lesssim
\frac{1}{n}\|\varphi\|_{p}.
$$

Hence, the assertion of the theorem follows by using the uniform boundedness of the operator $\mathcal{E}^{-1},\mathcal{P}_n: H^{p+2}[0,2\pi]^2\rightarrow H^{p+2}[0,2\pi]^2$.
\end{proof}

\begin{theorem}
\label{fullconvergence}
For sufficiently large $n$, the approximate equation \eqref{fullcollocation} is
uniquely solvable and the solution satisfies the error estimate
\begin{align}\label{errorestimate2}
\begin{split}
\|\widetilde\varphi^n-\varphi\|_p&\leq L\Big\{\|\mathcal{P}_n\mathcal{S}_0\mathcal{S}_0\varphi-\mathcal{S}_0\mathcal{S}_0\varphi\|_{p+2}+\|\mathcal{P}_n[\mathcal{E}_n^{-1}(\mathcal{J}_n+\mathcal{K}_n)-\mathcal{E}^{-1}(\mathcal{J}+\mathcal{K})]\varphi\|_{p+2}\Big\}\\
&\quad
+L\|\mathcal{P}_n[\mathcal{E}_n^{-1}\mathcal{A}_n-\mathcal{E}^{-1}\mathcal {A}]
w\|_{p+2},
\end{split}
\end{align}
where $L$ is a positive constant.
\end{theorem}

\begin{proof}
For all trigonometric polynomials $\varphi\in X_n^2$, it follows from Theorems
\ref{Kn-estimate} and \ref{Jn-estimate} that 
$$
\|\mathcal{P}_n[\mathcal{E}_n^{-1}(\mathcal{J}_n+\mathcal{K}_n)-\mathcal{E}^{-1}(\mathcal{J}+\mathcal{K})]\varphi\|_{p+2}\lesssim\frac{1}{n}\|\varphi\|_p\rightarrow0, \qquad n\rightarrow\infty
$$
for all $p>1/2$. Moreover, it is easy to see the estimates for
$\|(J_{1,n}^{\sigma}-J_{1}^{\sigma})\psi\|_{p+2}$ and
$\|(\widetilde{J}_{4,n}-\widetilde{J}_{4})\psi\|_{p+2}$  are valid analogously
as \eqref{Bestimate}. Then, we can obtain the uniform boundedness of the
operator
$\mathcal{P}_n[\mathcal{E}_n^{-1}(\mathcal{J}_n+\mathcal{K}_n)-\mathcal{E}^{-1}
(\mathcal{J}+\mathcal{K})]: H^p[0,2\pi]^2\rightarrow H^{p+2}[0,2\pi]^2$ from the
proofs of Theorems \ref{Kn-estimate} and  \ref{Jn-estimate}.
By the Banach--Steinhaus theorem (cf. \cite[Problem 10.1]{Kress2014-book}), we
get the pointwise convergence
\begin{align*}
\mathcal{P}_n[\mathcal{E}_n^{-1}(\mathcal{J}_n+\mathcal{K}_n)]
\varphi\rightarrow\mathcal{P}_n[\mathcal{E}^{-1}(\mathcal{J}+\mathcal{K})]
\varphi \quad \text{as}\quad n\rightarrow\infty
\end{align*}
for all $\varphi\in H^p[0,2\pi]^2$. 

Hence \eqref{errorestimate2} follows by employing \cite[Corollary
13.13]{Kress2014-book}, $\mathcal{E}_n^{-1}=\mathcal{E}^{-1}$ and the uniform
boundedness of the operator $\mathcal{E}_n^{-1},\mathcal{P}_n:
H^{p+2}[0,2\pi]^2\rightarrow H^{p+2}[0,2\pi]^2$.
\end{proof}

The above theorem implies that the full-discrete collocation method
\eqref{fullcollocation} converges in $H^p[0,2\pi]^2$ for each $p>1/2$.

%%%%%%%%%%%%%%%%%%%%%%%%%%%%%%%%%%%%%%%%%%%%%%%%%%%%%%%%%
%%%%%%%%%%%%%%%%%%%%%%%%%%%%%%%%%%%%%%%%%%%%%%%%%%%%%

\section{Numerical experiments}

In practice, instead of \eqref{fullcollocation}, we only need to solve the
equivalent full-discrete equation
\begin{align}\label{transformfull}
\mathcal{P}_n\mathcal{H}_n\widetilde\varphi^{n}+\mathcal{P}_n\mathcal{B}
_n\widetilde\varphi^{n}=\mathcal{P}_nw 
~\Leftrightarrow~
A_n\widetilde\varphi^{n}=w^{n},
\end{align}
where $A_n$ is a coefficient matrix of the full-discrete equation. In fact, suppose that $\beta(t)=\frac{\kappa_{\mathfrak p}^2+\kappa_{\mathfrak s}^2}{8\pi^2}|z'(t)|^2$, the equation \eqref{fullcollocation} is equivalent to
\begin{align}
&\mathcal{P}_n(\mathcal{S}_{0,n}\mathcal{S}_{0,n})\widetilde\varphi^n+\mathcal{P}_n[\beta^{-1}(\mathcal{J}_n+\mathcal{K}_n)\widetilde\varphi^n]=\mathcal{P}_n[\beta^{-1}\mathcal{A}_nw] 
\nonumber\\
~\Leftrightarrow~&
\mathcal{P}_n[\beta(\mathcal{S}_{0,n}\mathcal{S}_{0,n}\widetilde\varphi^n)]+\mathcal{P}_n[(\mathcal{J}_n+\mathcal{K}_n)\widetilde\varphi^n]=\mathcal{P}_n(\mathcal{A}_nw)
\nonumber\\
\label{equivalent eqn}
~\Leftrightarrow~&
A_n^2\widetilde\varphi^n=A_nw^n.
\end{align}
Since \eqref{fullcollocation} is uniquely solvable by Theorem
\ref{fullconvergence}, it implies that the matrix $A_n^2$ is invertible, i.e.,
$\det{A}_n^2=(\det{A}_n)^2\not=0$, and consequently ${A}_n$ is invertible.
Hence, \eqref{fullcollocation} is equivalent to \eqref{transformfull} by
multiplying matrix ${A}_n^{-1}$ on both ends of the equation \eqref{equivalent
eqn}. It is worth mentioning that the equivalent full-discrete equation
\eqref{transformfull} is extremely efficient since it is established via simple
quadrature operators $\mathcal{H}_n$ and $\mathcal{B}_n$.

For the smooth integrals, we simply use the trapezoidal rule 
\begin{align*} 
%\label{traperule}
\int_{0}^{2\pi}f(\varsigma)\mathrm{d}\varsigma\approx\frac{\pi}{n}\sum_{j=0}^{
	2n-1}f(\varsigma_j^{(n)}).
\end{align*}
For the singular integrals, we employ the following
quadrature rules via the trigonometric interpolation:
\begin{align}\label{quadrature}
\begin{split} 
\int_{0}^{2\pi}\ln\Big(4\sin^2\frac{t-\varsigma}{2}\Big)f(\varsigma)\,\mathrm{d}
\varsigma&\approx\sum_{j=0}^{2n-1}R_j^{(n)}(t)f(\varsigma_j^{(n)}),
\\
\frac{1}{2\pi}\int_0^{2\pi}\cot\frac{\varsigma-t}{2}f(\varsigma)\,\mathrm{d}
\varsigma&\approx\sum_{j=0}^{2n-1}U_j^{(n)}(t)f(\varsigma_j^{(n)}),
\\
\int_0^{2\pi}\ln\Big(4\sin^2\frac{t-\varsigma}{2}
\Big)\sin(t-\varsigma)f(\varsigma)\,\mathrm{d}\varsigma&\approx\sum_{j=0}^{2n-1}
V_j^{(n)}(t)f(\varsigma_j^{(n)}),
\end{split}
\end{align}
where the quadrature weights are given by
\begin{align*}%\label{weight_R}
R_j^{(n)}(t)&=-\frac{2\pi}{n}\sum_{m=1}^{n-1}\frac{1}{m}\cos\Big[m(t-\varsigma_j^{(n)})\Big]
-\frac{\pi}{n^2}\cos\Big[n(t-\varsigma_j^{(n)})\Big], \\
%%%%%%%%%%%%%%%%%%%%%%%%%%%%%%%%%%%%%%%%%%%%%%%%%%%
U_j^{(n)}(t)&=\frac{1}{2n}\big[1-\cos n(\varsigma_j^{(n)}-t)\big]\cot\frac{\varsigma_j^{(n)}-t}{2}, \\
%%%%%%%%%%%%%%%%%%%%%%%%%%%%%%%%%%%%%%%%%%%%%%%%%%%%%
V_j^{(n)}(t)&=-\frac{\pi}{2n}\sin(\varsigma_j^{(n)}-t)+\frac{2\pi}{n}\sum_{m=2}^{n-1}\frac{\sin\Big[m(\varsigma_j^{(n)}-t)\Big]}{m^2-1}+\frac{2\pi\sin\Big[n(\varsigma_j^{(n)}-t)\Big]}{n(n^2-1)}.
\end{align*}
Here, the weight $V_j^{(n)}$ is calculated by using \cite[Lemma
8.23]{Kress2014-book} and we also refer to \cite{Kress2014-book} for the weights
$R_j^{(n)}$ and $U_j^{(n)}$. On the other hand, the last items of
$\mathcal{H}_{1,n}$ and $\mathcal{H}_{2,n}$ can be offset by the following item 
$$
\int_0^{2\pi}\mathcal{P}_n \Bigg\{\left[\begin{array}{cc}
0& \widetilde{h}_3^{\mathfrak s}(t,\cdot)-h_3^{\mathfrak s}(t,\cdot) \\ 
\widetilde{h}_3^{\mathfrak p}(t,\cdot)-h_3^{\mathfrak p}(t,\cdot)& 0
\end{array}
\right]\chi\Bigg\}(\varsigma)\,\mathrm{d}\varsigma
$$ 
in $\mathcal{B}_{2,n}$. Thus, the equation \eqref{transformfull} becomes
\begin{align}\label{numfull}
\begin{split}
w_{1,i}^{(n)} &= -\varphi^{(n)}_{1,i}+
\sum_{j=0}^{2n-1}X_{ij,{\mathfrak
p}}^{(n)}\varphi_{1,j}^{(n)}+\sum_{j=0}^{2n-1}Y_{ij,{\mathfrak
s}}^{(n)}\varphi_{2,j}^{(n)},\\
w_{2,i}^{(n)} &= \varphi^{(n)}_{2,i}+
\sum_{j=0}^{2n-1}Y_{ij,{\mathfrak
p}}^{(n)}\varphi_{1,j}^{(n)}-\sum_{j=0}^{2n-1}X_{ij,{\mathfrak
s}}^{(n)}\varphi_{2,j}^{(n)},
\end{split}
\end{align}
where $w_{l,i}^{(n)}=w_l(\varsigma_i^{(n)}),
\varphi_{l,i}^{(n)}=\varphi_l(\varsigma_i^{(n)})$ for $i,j=0,\cdots,2n-1$, $l=1,
2$, and
\begin{align*}
X_{ij,\sigma}^{(n)}&=R_j^{(n)}(\varsigma_i^{(n)})k_1^\sigma(\varsigma_i^{(n)},
\varsigma_j^{(n)})+\frac{\pi}{n}k_2^\sigma(\varsigma_i^{(n)},\varsigma_j^{(n)}),
\\
Y_{ij,\sigma}^{(n)}&=U_j^{(n)}(\varsigma_i^{(n)})
+\frac{\kappa_{\sigma}^2}{4\pi}|z'(\varsigma_i^{(n)})|^2V_j^{(n)}(\varsigma_i^{(n)})+
R_j^{(n)}(\varsigma_i^{(n)})\widetilde{h}_2^\sigma(\varsigma_i^{(n)},\varsigma_j^{(n)})\\
&\quad+\frac{\pi}{n}h_3^\sigma(\varsigma_i^{(n)},\varsigma_j^{(n)})+\frac{\pi}{
n } \widetilde{h}_1(\varsigma_i^{(n)},\varsigma_j^{(n)}).
\end{align*}

\begin{remark}\label{reduceY}
Moreover, a straightforward calculation yields
$$
V_j^{(n)}(t)-R_j^{(n)}(t)\sin(t-\varsigma_j^{(n)})=\frac{\pi\sin n(t-\varsigma_j^{(n)})}{n(n+1)}+\frac{\pi}{n^2}\sin n(t-\varsigma_j^{(n)})\cos(t-\varsigma_j^{(n)}),
$$
which implies $V_j^{(n)}(\varsigma_i^{(n)})-R_j^{(n)}(\varsigma_i^{(n)})\sin(\varsigma_i^{(n)}-\varsigma_j^{(n)})=0$. Therefore, $Y_{ij,\sigma}^{(n)}$ can be reduced to
$$
Y_{ij,\sigma}^{(n)}=U_j^{(n)}(\varsigma_i^{(n)})
+R_j^{(n)}(\varsigma_i^{(n)})h_2^\sigma(\varsigma_i^{(n)},\varsigma_j^{(n)})+\frac{\pi}{n}h_3^\sigma(\varsigma_i^{(n)},\varsigma_j^{(n)})+\frac{\pi}{n}\widetilde{h}_1(\varsigma_i^{(n)},\varsigma_j^{(n)}).
$$
From this, we find that $E^\sigma H_2\varphi$ in \eqref{decomposition} is only
used for the theoretical analysis. 
\end{remark}

\begin{remark}
From \cite[Section 4]{KirschRitter1999}, we know that the trapezoidal rule and
the quadrature formulas \eqref{quadrature} yield convergence of exponential
order for periodic analytic function $f$. In addition, from \cite[Theorem
11.7]{Kress2014-book}, we conclude exponential convergence of our method if the
boundary of obstacle and the exact solution are analytic.
\end{remark}

\begin{table} 
\caption{Parametrization of the exact boundary curves.}
\label{boundary}
\begin{tabular}{lll}
\toprule[1pt]
Type           &Parametrization\\
\midrule  
Apple-shaped   & $z(t)=\displaystyle\frac{0.55(1+0.9\cos{t}+0.1\sin{2t})}{1+0.75\cos{t}}(\cos{t}, \sin{t}), \quad t\in [0,2\pi]$  \vspace{1.5ex} \\
Peach-shaped  & 
$z(t)=0.22(\cos^2{t}\sqrt{1-\sin{t}}+2)(\cos{t}, \sin{t}), \quad t\in[0,2\pi]$
\vspace{1ex}\\ 
Drop-shaped  & 
$\displaystyle z(t)=(2\sin{\frac{t}{2}}-1, -\sin{t}), \quad t\in[0,2\pi]$ \vspace{1ex}\\
Heart-shaped  & 
$\displaystyle z(t)=(\frac{3}{2}\sin{\frac{3t}{2}}, \sin{t}), \quad t\in[0,2\pi]$\\
\bottomrule[1pt]
\end{tabular}
\end{table}

\begin{table}
\centering 
\caption{Numerical errors for the apple-shaped and peach-shaped obstacles with
$\omega=\pi$.} 
\label{numerror2} 
\begin{tabular}{c|c|c|c|c}  
\toprule[1pt]
& \multicolumn{2}{c|}{Apple-shaped} & \multicolumn{2}{c}{Peach-shaped}  \\ 
\cline{2-5}
$n$&$\|\phi_*-\phi^{(n)}\|_{L^2}$    &$\|\psi_*-\psi^{(n)}\|_{L^2}$
&$\|\phi_*-\phi^{(n)}\|_{L^2}$ &$\|\psi_*-\psi^{(n)}\|_{L^2}$ \\
\hline
8& 0.0677 & 0.0613 &0.0044&0.0050\\
16&2.1192e-04&1.6939e-04&3.9734e-04&4.5298e-04 \\
32&3.7880e-07&3.0432e-07&5.4337e-05&6.1341e-05 \\
64&6.6341e-12&5.2998e-12&6.9918e-06&7.8543e-06 \\
128&1.6200e-15&1.6162e-15&8.8584e-07&9.9062e-07 \\
256&1.9389e-15&2.2955e-15&1.1140e-07&1.2420e-07 \\
512&3.1540e-15&2.9617e-15&1.3962e-08&1.5539e-08 \\
1024&4.2380e-15&3.7504e-15&1.7475e-09&1.9429e-09 \\
\bottomrule[1pt] 
\end{tabular}
\end{table}

\begin{table}
\centering 
\caption{Numerical errors for the apple-shaped and peach-shaped obstacles with
$\omega=100\pi$.} 
\label{numerror3} 
\begin{tabular}{c|c|c|c|c} 
\toprule[1pt]	 
& \multicolumn{2}{c|}{Apple-shaped} & \multicolumn{2}{c}{Peach-shaped}  \\ 
\cline{2-5}
$n$&$\|\phi_*-\phi^{(n)}\|_{L^2}$    &$\|\psi_*-\psi^{(n)}\|_{L^2}$
&$\|\phi_*-\phi^{(n)}\|_{L^2}$ &$\|\psi_*-\psi^{(n)}\|_{L^2}$  \\  
\hline
64&2.2192&1.1012&5.9298&2.5847 \\
128&7.2908e-02&9.2983e-02&1.0250e-01&1.0459e-01 \\
256&5.5220e-07&1.0522e-06&4.1347e-07&9.4716e-07 \\
512&6.0630e-13&4.4848e-13&5.0089e-08&3.6659e-08 \\
1024&5.3276e-13&3.7980e-13&6.2029e-09&4.3889e-09 \\
2048&4.8503e-13&4.0631e-13&7.7281e-10&5.3866e-10 \\
4096&5.3277e-13&3.9049e-13&9.6473e-11&6.6741e-11 \\
\bottomrule[1pt]
\end{tabular}
\end{table}

\begin{figure}
\centering 
\subfigure[$\Re(\phi)$]
{\includegraphics[width=0.45\textwidth]{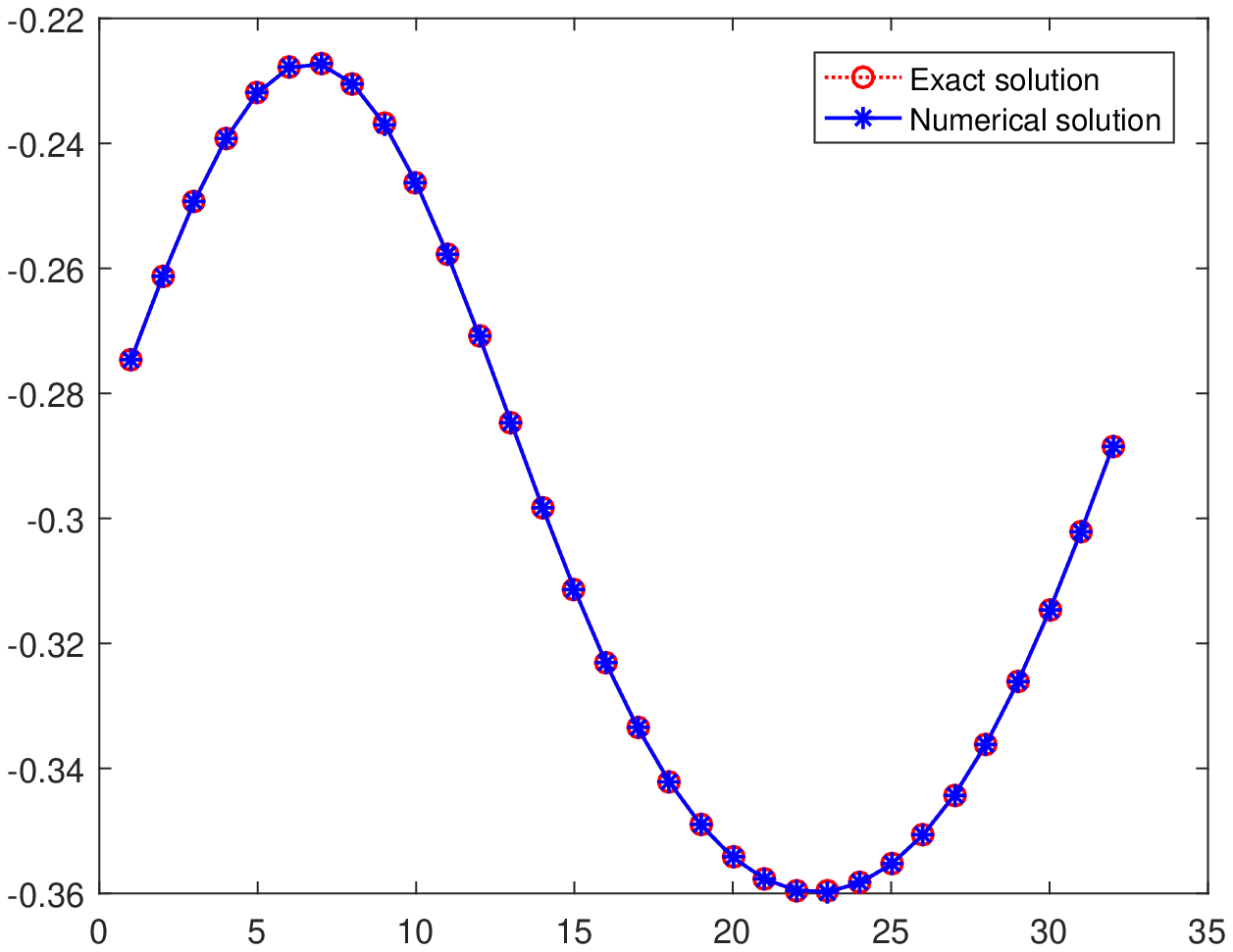}}
\subfigure[$\Im(\phi)$]
{\includegraphics[width=0.45\textwidth]{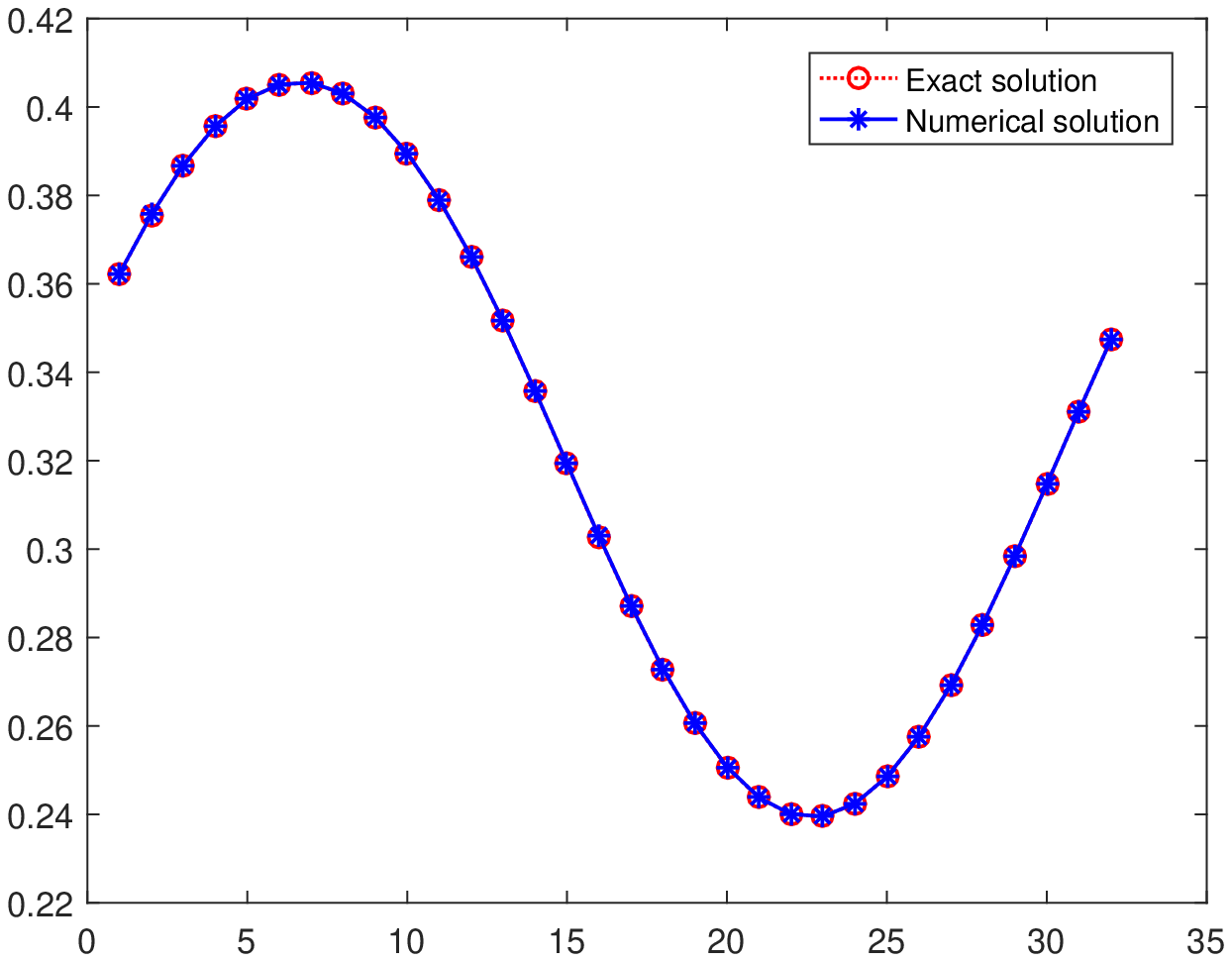}} 
\subfigure[$\Re(\psi)$]
{\includegraphics[width=0.45\textwidth]{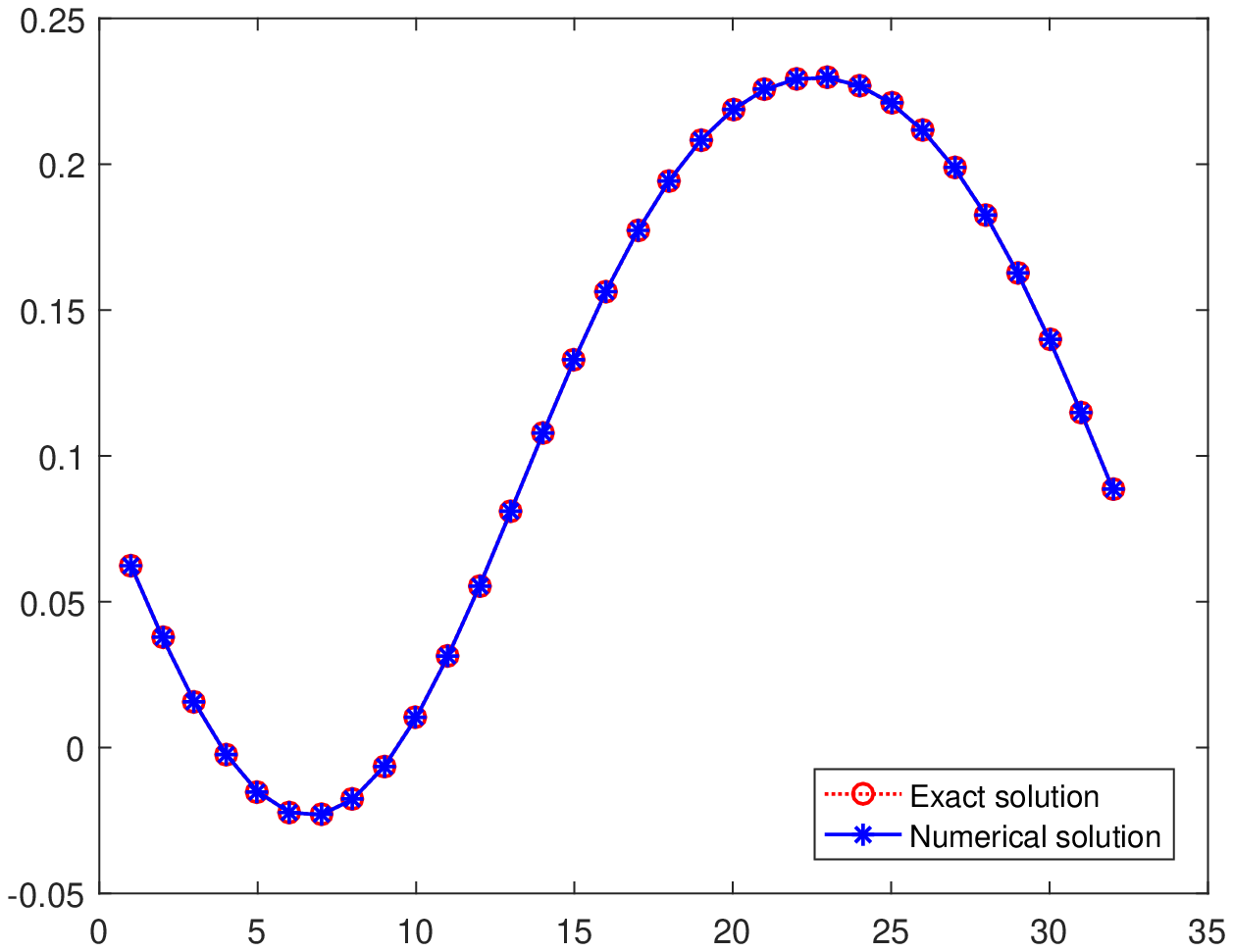}}
\subfigure[$\Im(\psi)$]
{\includegraphics[width=0.45\textwidth]{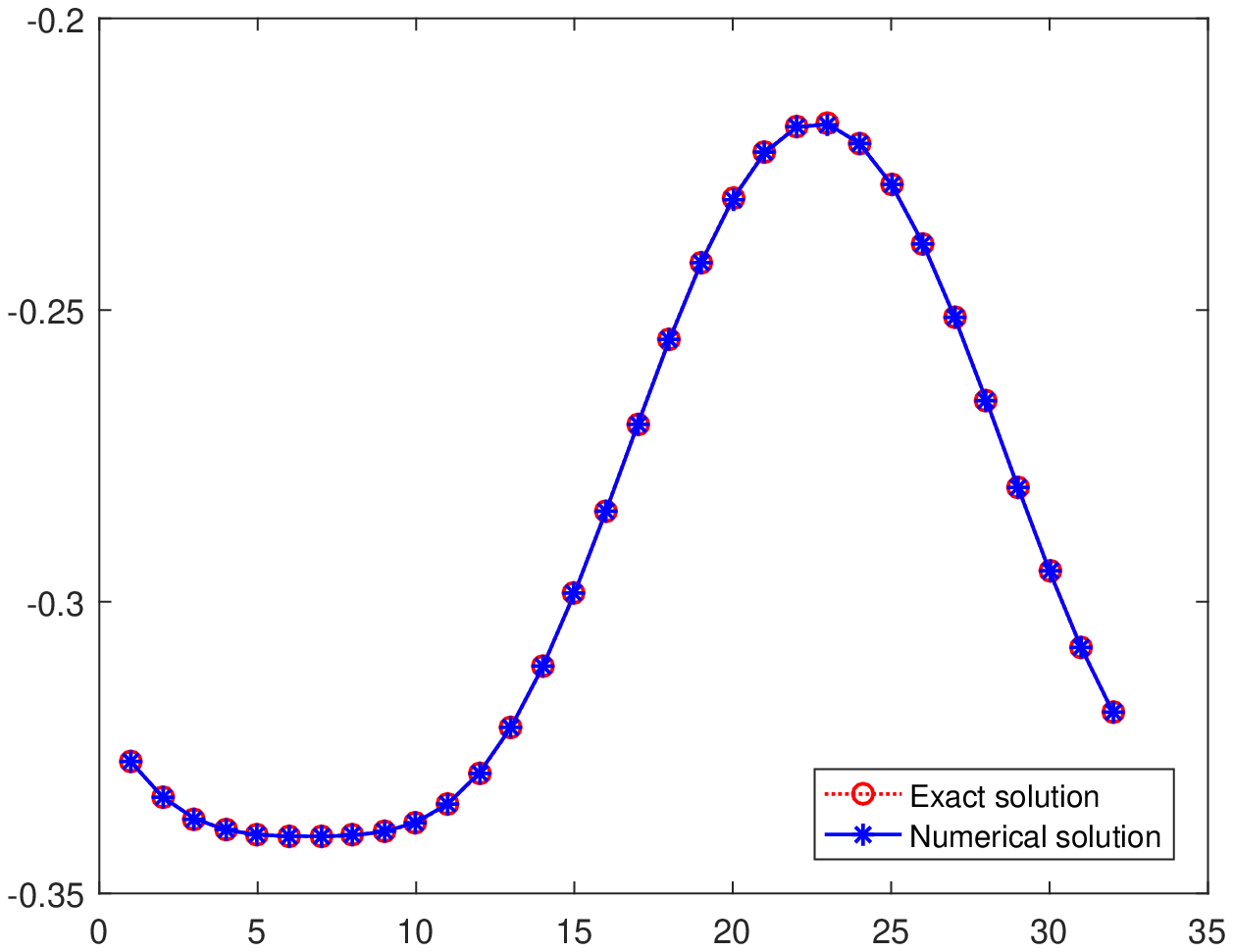}}
\caption{Numerical solutions and corresponding exact solutions for the
apple-shaped obstacle with $\omega=\pi, n=\tilde{n}=16$.}\label{solutions1pi}
\end{figure}

\begin{figure}
\centering 
\subfigure[$\Re(\phi)$]
{\includegraphics[width=0.45\textwidth]{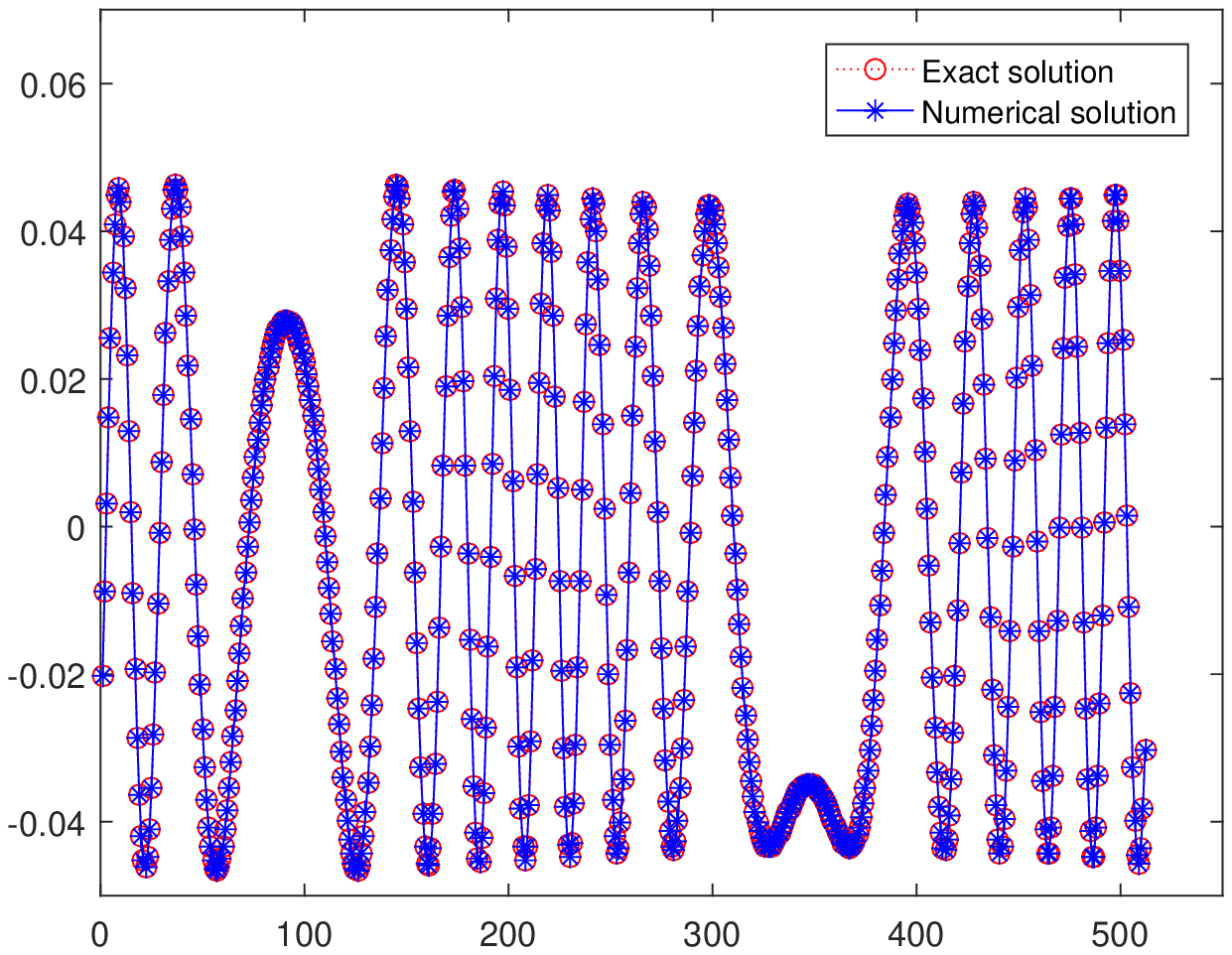}}
\subfigure[$\Im(\phi)$]
{\includegraphics[width=0.45\textwidth]{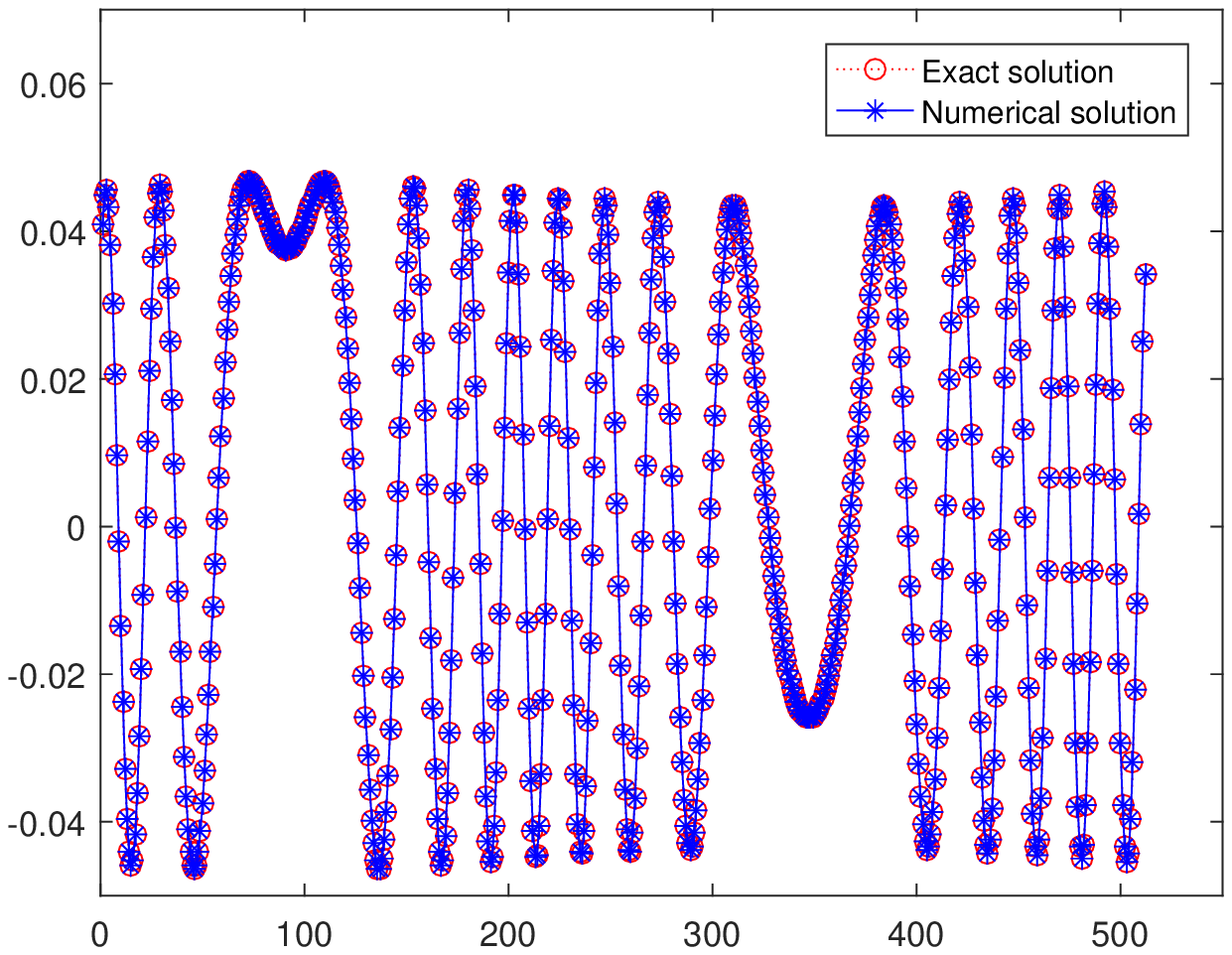}} 
\subfigure[$\Re(\psi)$]
{\includegraphics[width=0.45\textwidth]{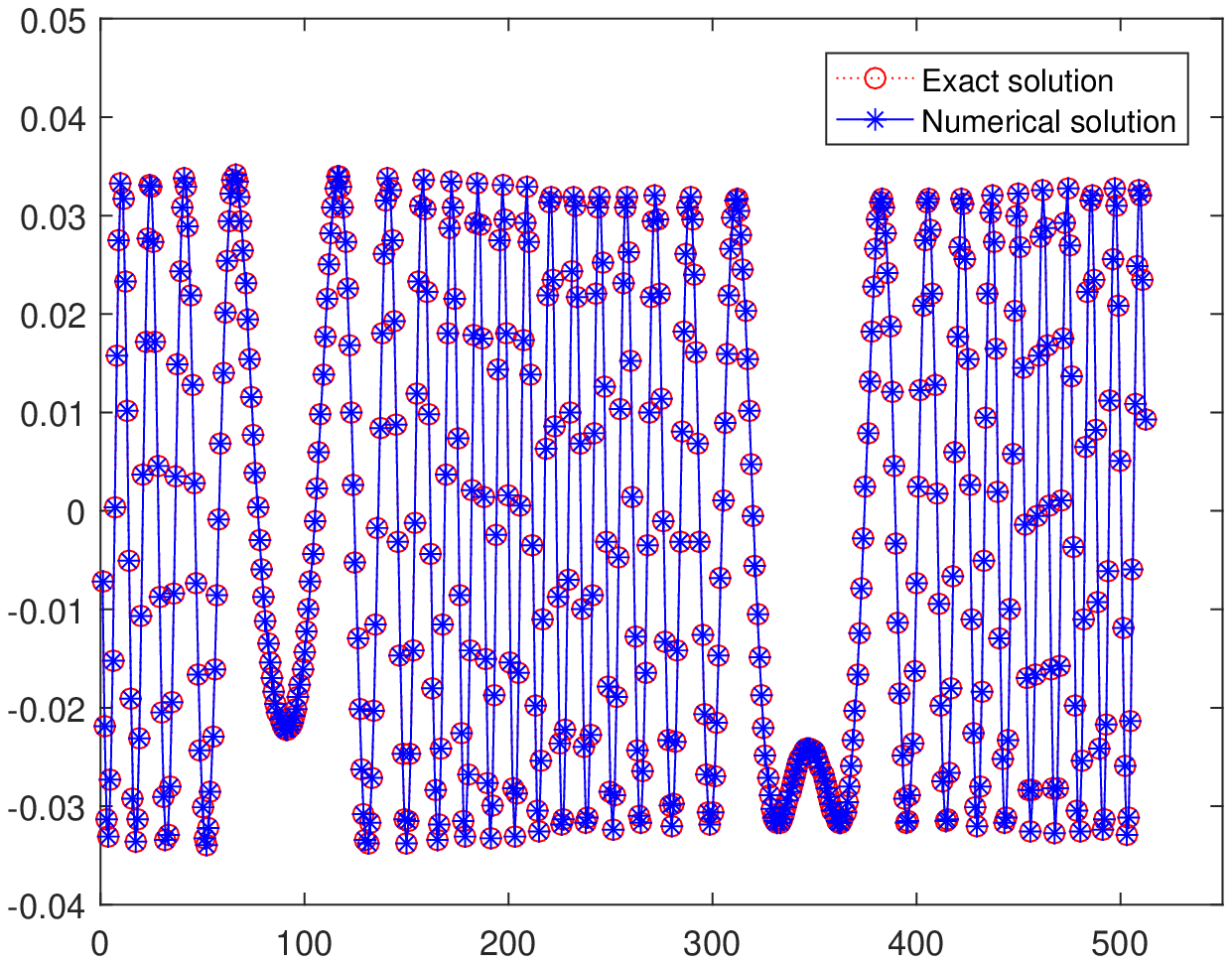}}
\subfigure[$\Im(\psi)$]
{\includegraphics[width=0.45\textwidth]{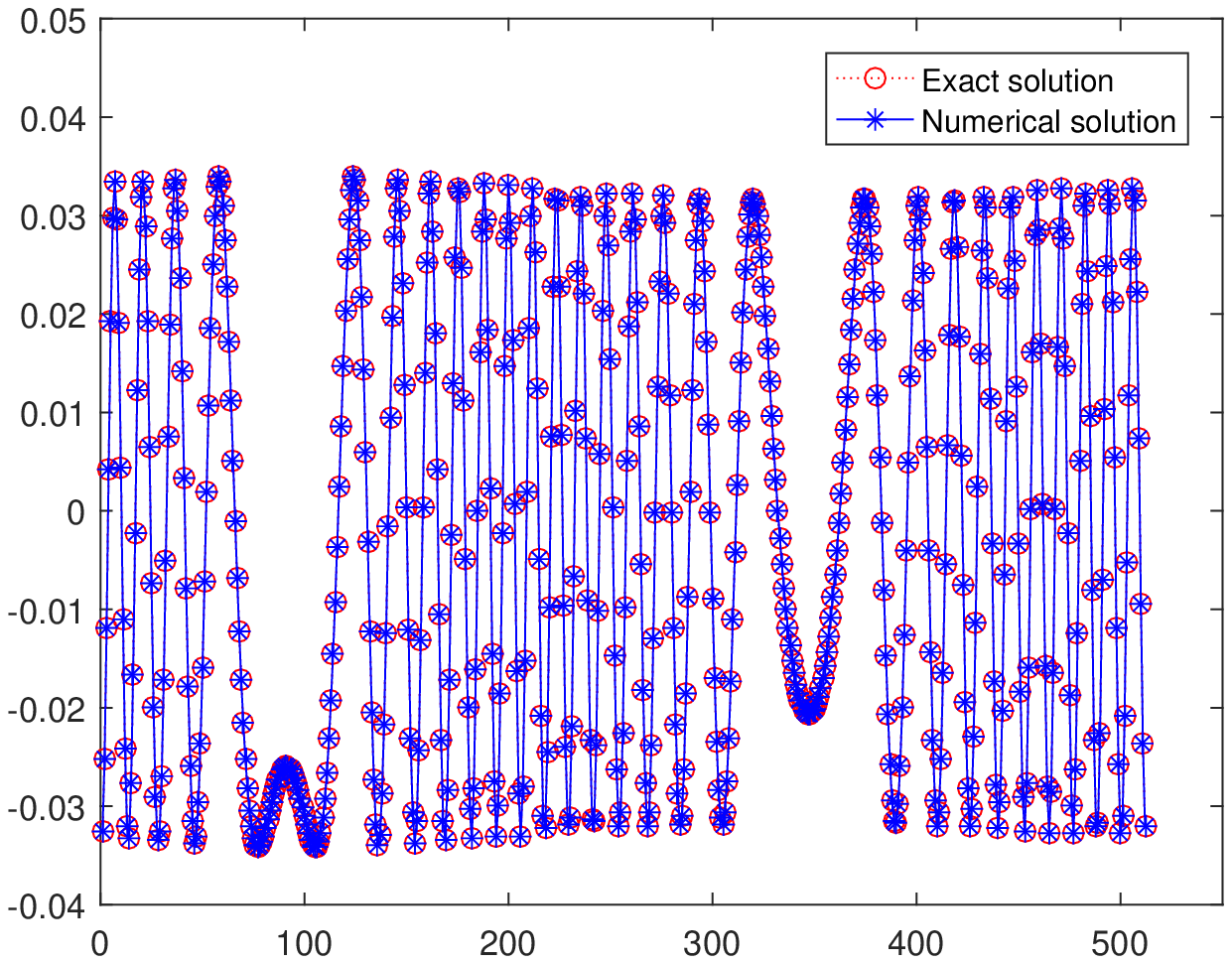}}
\caption{Numerical solutions and corresponding exact solutions for
the apple-shaped obstacle with $\omega=100\pi,
n=\tilde{n}=256$.}\label{solutions100pi}
\end{figure}

\subsection{Numerical examples: smooth obstacles}

In this subsection, we consider the elastic scattering by an apple-shaped and a
peach-shaped obstacle with analytic and $\mathcal{C}^2$ boundary, respectively,
The parametrizations of these two boundary curves are given in Table
\ref{boundary}. To test the accuracy of the trigonometric collocation method, we
construct an exact solution by letting the exterior field of the elastic
obstacle be generated by two point sources located at 
$\bar{x}=(0.1,0.2)^\top\in D$, i.e.,
\begin{align}\label{exact solution}
\phi_*(x)=H_0^{(1)}(\kappa_{\mathfrak p}|x-\bar{x}|),\quad \psi_*(x)=H_0^{(1)}(\kappa_{\mathfrak s}|x-\bar{x}|),\quad x\in\mathbb{R}^2\setminus\overline{D}.
\end{align}
Due to the uniqueness of the boundary value problem \eqref{HelmholtzDec}, the
solution can be constructed explicitly by enforcing the following boundary
conditions on $\Gamma_D$:
\begin{align*}
f_1=\partial_\nu\phi_*+\partial_\tau\psi_*,\qquad
f_2=\partial_\tau\phi_*-\partial_\nu\psi_*. 
\end{align*}
In numerical experiments, we take the Lam\'{e} parameters $\lambda=3.88,
\mu=2.56$ and let the observation points be generated by
$\{\varsigma_i^{(n)}\}_{i=0}^{2\tilde{n}-1}$, $\tilde{n}=16$ are distributed on
a circle $\partial B=\{x\in\mathbb{R}^2: |x|=3\}$. We list the numerical errors
between the numerical solution and the corresponding exact solution with
$L^2(\partial B)$ norm in Tables \ref{numerror2} and \ref{numerror3} for
the angular frequency $\omega=\pi$ and $\omega=100\pi$, respectively. It can be
easily seen from the results that the accuracy is improved dramatically as the
number of collocation points are increased. In fact, our method has an
exponential convergence which confirms the theoretical analysis. We also find
that the convergence rate of the apple-shaped obstacle with analytic boundary is
faster than that of the peach-shaped obstacle with $\mathcal{C}^2$ boundary. The
numerical solution and the corresponding exact solution are shown in Figure
\ref{solutions1pi} for the apple-shaped obstacle. Clearly they coincide
perfectly with $\omega=\pi, n=16$.

For the high-frequency case, we can get the same highly accurate results as
those of the low-frequency case by increasing the number of interpolation
points. The numerical solution and the corresponding exact solution are shown in
Figure \ref{solutions100pi} for the apple-shaped obstacle with $\omega=100\pi$.
As can be seen, the numerical solutions and the exact solutions also coincide
perfectly when $n=\tilde{n}=256$.

It is worth mentioning that for a given incident wave and elastic obstacle, in
view of \eqref{behaviour relation} and \eqref{farfield}, together with
\eqref{numfull}, we can get the compressional and shear far-field patterns
immediately by using the trapezoidal rule. With the aid of \eqref{HelmDeco},
\eqref{singlelayer} and \eqref{numfull}, noting $\boldsymbol v_{\mathfrak
p}=\nabla\phi$ and $\boldsymbol v_{\mathfrak s}=\boldsymbol{\rm curl}\psi$, we
can also easily obtain the compressional and shear elastic scattered fields by
using the trapezoidal rule, too, if the test points are not too close to the
boundary.

%%%%%%%%%%%%%%%%%%%%%%%%%%%%%%%%%%%%%%%%%%%%%%%%%%%%%%%%%
\begin{figure}
\centering 
\subfigure[$p=2, n=16$]
{\includegraphics[width=0.45\textwidth]{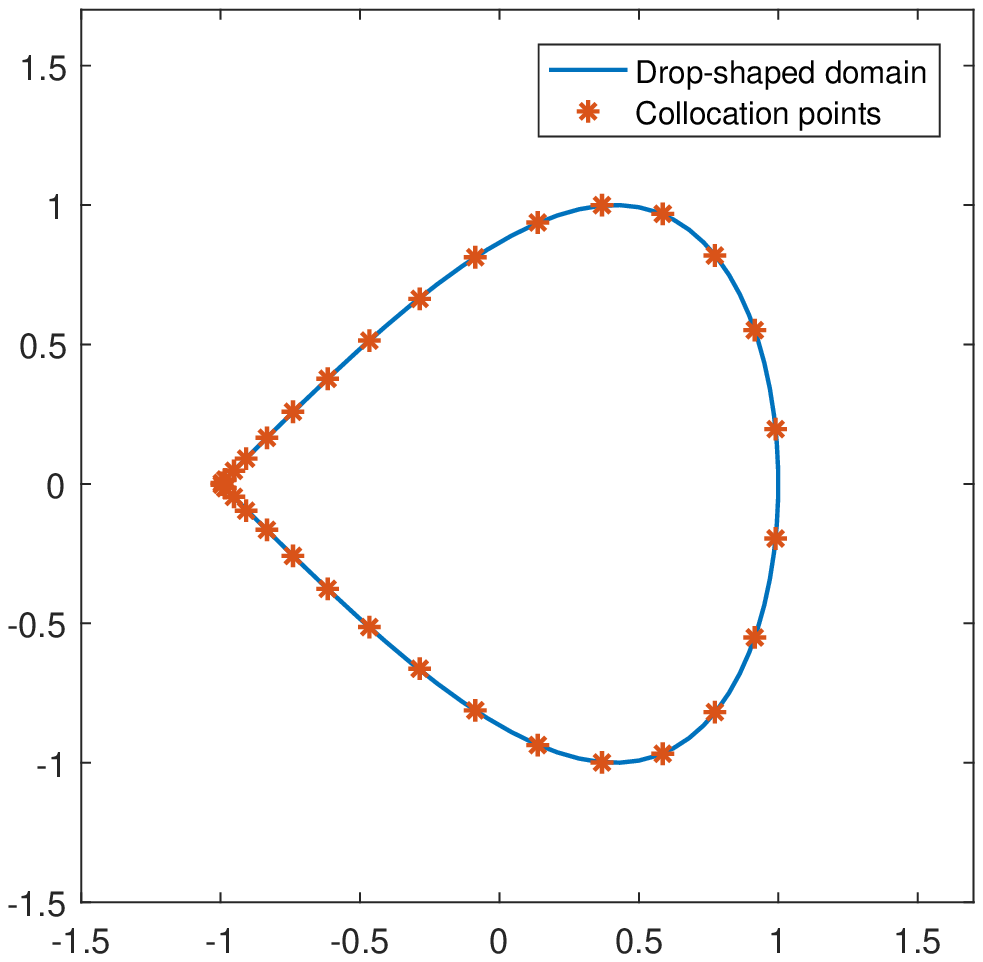}}
\subfigure[$p=2, n=16$]
{\includegraphics[width=0.45\textwidth]{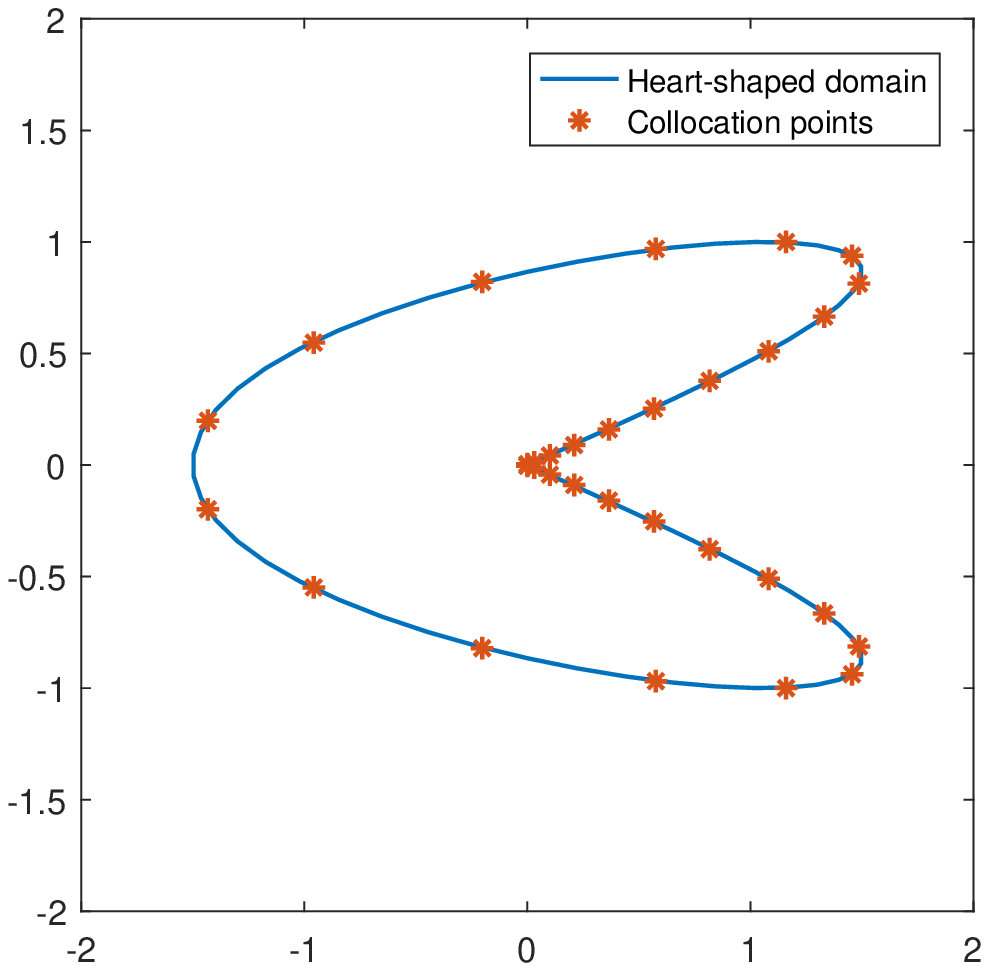}} 
\caption{Collocation points on the drop-shaped and heart-shaped
obstacles.}\label{points}
\end{figure}

\begin{table}
\centering 
\caption{Numerical errors for the drop-shaped domain with $\omega=\pi$.} 
\label{numerror6} 
\begin{tabular}{c|c|c|c|c}  
\toprule[1pt]
& \multicolumn{2}{c|}{Point source} & \multicolumn{2}{c}{Plane wave}  \\ 
\cline{2-5}
$n$&$\|\phi_*-\phi^{(n)}\|_{L^2}$    &$\|\psi_*-\psi^{(n)}\|_{L^2}$
&$\|\phi^{(n_*)}-\phi^{(n)}\|_{L^2}$ &$\|\psi^{(n_*)}-\psi^{(n)}\|_{L^2}$ \\
\hline
16&2.0999e-03&2.1650e-03&4.5396e-01&5.8469e-01 \\
32&2.4347e-08&3.1562e-08&3.9571e-03&4.9919e-03 \\
64&1.2669e-08&1.6911e-08&2.7795e-04&3.7244e-04 \\
128&3.8572e-10&5.1477e-10&2.4339e-05&2.8050e-05 \\
256&1.1736e-11&1.5791e-11&1.1917e-04&1.6019e-04 \\
512&1.3430e-14&1.7613e-14&9.2866e-06&1.2593e-05 \\
1024&6.2155e-15&4.9608e-15&5.5606e-06&7.5428e-06 \\
2048&6.1235e-15&5.9362e-15&1.1380e-07&1.5203e-07 \\
\bottomrule[1pt] 
\end{tabular}
\end{table}

\begin{table}
\centering 
\caption{Numerical errors for the heart-shaped domain with $\omega=\pi$.} 
\label{numerror7} 
\begin{tabular}{c|c|c|c|c}  
\toprule[1pt]
& \multicolumn{2}{c|}{Point source} & \multicolumn{2}{c}{Plane wave}  \\ 
\cline{2-5}
$n$&$\|\phi_*-\phi^{(n)}\|_{L^2}$    &$\|\psi_*-\psi^{(n)}\|_{L^2}$
&$\|\phi^{(n_*)}-\phi^{(n)}\|_{L^2}$ &$\|\psi^{(n_*)}-\psi^{(n)}\|_{L^2}$ \\
\hline
16&4.3673e-02&1.0523e-01&6.0929e-03&1.1896e-02 \\
32&6.0075e-04&1.5144e-03&2.5014e-05&3.7337e-05 \\
64&4.3721e-07&8.4011e-07&1.2432e-07&1.2820e-07 \\
128&1.8692e-09&1.2039e-09&2.2355e-09&1.4398e-09 \\
256&1.1752e-11&7.5236e-12&1.3571e-11&8.6880e-12 \\
512&1.6306e-13&1.0433e-13&1.8443e-13&1.1843e-13 \\
1024&4.5946e-15&4.0851e-15&8.0003e-15&5.9002e-15 \\
2048&1.1742e-14&1.3828e-14&9.6079e-15& 1.0804e-14 \\
\bottomrule[1pt] 
\end{tabular}
\end{table}

\subsection{Numerical examples: nonsmooth obstacles}

In this subsection, we assume that $D$ has a single corner at $x_0$ and assume
$\Gamma_D\setminus \{x_0\}$ to be analytic. The angle $\gamma$ at the corner is
supposed to satisfy $0<\gamma<2\pi$. Suppose that the corner point $x_0$
corresponds to the parameter $t=0$ in the parametric representation of
$\Gamma_D$. To test the accuracy of our method, we adopt the exact solutions in
form of \eqref{exact solution} with the point source located
at $\bar{x}=(0.1,0.2)^\top$ and $\bar{x}=(-0.5,0.2)^\top$ for the drop-shaped
and heart-shaped obstacles, respectively. The interior angles are
$\gamma=\pi/2$ and $\gamma=3\pi/2$ for the drop-shaped
and heart-shaped obstacles, respectively. The parameterizations of
these two boundary curves are also shown in Table \ref{boundary}. In addition,
we consider the case that the obstacle is illuminated by a compressional plane
wave $\boldsymbol{u}^{\rm inc}$ which is given by
\[
\boldsymbol{u}^{\rm inc}(x)=d \mathrm{e}^{{\rm i} \kappa_{\mathfrak p}d\cdot x},
\]
where $d=(\cos\theta, \sin\theta)^\top$ is the unit propagation direction
vector. 

To resolve the field near the corner, we adopt the graded mesh by taking 
the substitution \cite{DR-shu2,Kress1990} $t=w(s)$, which is given by
\begin{align*}
w(s)=2\pi\frac{[v(s)]^p}{[v(s)]^p+[v(2\pi-s)]^p}, \qquad 0\leq s\leq 2\pi,
\end{align*}
where
\begin{align*}
v(s)=\Big(\frac{1}{p}-\frac{1}{2}\Big)\Big(\frac{\pi-s}{\pi}\Big)^3+\frac{1}{p}
\frac{s-\pi}{\pi}+\frac{1}{2}, \qquad p\geq2,
\end{align*}
and is applied to the parametric curve of the drop-shaped and heart-shaped
obstacles. In experiments, we choose $s_j:=\pi j/n+\pi/(2n)$ as the
collocation points in \eqref{numfull}. The generated points $w(s_j)$,
$j=0,\cdots,2n-1$ of the graded mesh on the both boundaries are presented in
Figure \ref{points} for $p=2$.

The numerical errors between the numerical solution and the exact solution
\eqref{exact solution} with $L^2(\partial B)$ norm for the drop-shaped and
heart-shaped obstacles are listed in Tables \ref{numerror6} and \ref{numerror7}
with the angular frequency $\omega=\pi$ and $\tilde{n}=16$. 
Additionally, we calculate the values of compressional and shear scattered
fields $\phi^{(n_*)}, \psi^{(n_*)}$ on $\partial B$ with $\tilde{n}=16,
n_*=4096$ by the incident plane wave with $\theta=\pi/6$, and compare them with
the cases of other numbers of collocation points. For the point source case, the
solver quickly converges to machine precision for both domains. This is due to
the analyticity of the artificial solution. On the other hand, for the true
scattering problem, i.e., the scattering problem of the plane wave incidence, we
note that the numerical error of the heart-shaped domain is better than that of
the drop-shaped domain. The reason is apparently related to the concavity of the
domain. Detailed analysis will be investigated in 
the future work.

\section{Conclusion}

We have proposed a novel boundary integral formulation and developed a highly
accurate numerical method for solving the time-harmonic elastic scattering from
a rigid bounded obstacle immersed in a homogeneous and isotropic elastic medium.
Using the Helmholtz decomposition, we reduce the scattering problem to a coupled
boundary integral equation with singular integral operators. By introducing an
appropriate reqularizer to the coupled system, we split the operator equation in
the form of an isomorphic operator plus a compact one. The convergence is shown
for both the semi-discrete and full-discrete schemes via the trigonometric
collocation method. Numerical experiments for smooth and nonsmooth obstacles,
especially for the obstacles with corners, are presented to demonstrate the
superior performance of the proposed method. Along this line, we intend to
extend the current work to the coupled fluid-solid scattering problem and the
three-dimensional elastic obstacle scattering problem, where the more
complicated model equations need to be considered.

\end{document}